\documentclass[10pt]{article}
\usepackage{amsmath,amsthm,amscd,amssymb,mathrsfs,setspace}
\usepackage{latexsym,epsf,epsfig}
\usepackage{epstopdf}
\usepackage[hmargin=3.5cm,vmargin=3.5cm]{geometry}
\usepackage{color}
\newcommand{\ds}{\displaystyle}

\newcommand{\reals}{\mathbb{R}}
\newcommand{\realstwo}{\mathbb{R}^2}
\newcommand{\realsthree}{\mathbb{R}^3}
\newcommand{\xb}{{\bf{x}}}

\newcommand{\Dx}{{\partial_x}}
\newcommand{\pd}{{\partial}}
\newcommand{\cA}{{\mathscr{A}}}
\newcommand{\Dn}{\partial_{\nu}}

\newcommand{\cE}{{\mathcal{E}}}
\newcommand{\cD}{\mathscr{D}}
\newcommand{\bA}{\mathbb{A}}
\newcommand{\bP}{\mathbb{P}}
\newcommand{\bX}{\mathbb{X}}
\newcommand{\cF}{\mathscr{F}}

\newcommand{\R}{\mathbb{R}}
\newcommand{\cO}{\mathcal{O}}
\newcommand{\cU}{\mathcal{U}}
\newcommand{\Om}{{\Omega}}
\newcommand{\om}{{\omega}}
\newcommand{\al}{{\alpha}}
\newcommand{\ga}{{\gamma}}

\newcommand{\la}{{\lambda}}

\theoremstyle{plain}
\newtheorem{theorem}{Theorem}[section]
\newtheorem{lemma}[theorem]{Lemma}
\newtheorem{proposition}[theorem]{Proposition}

\newtheorem{assumption}{Assumption}[section]
\theoremstyle{remark}
\newtheorem{remark}{Remark}[section]
\numberwithin{equation}{section}
\numberwithin{theorem}{section}
\numberwithin{remark}{section}

\title{Evolution Semigroups  in Supersonic Flow-Plate Interactions}
\date{\today}

 \author{\begin{tabular}[t]{c@{\extracolsep{1em}}c@{\extracolsep{1em}}c}
           Igor Chueshov  & Irena Lasiecka & Justin T. Webster \\
\it        Kharkov National Univ. & \it Univ. of Virginia & \it Univ. of Virginia \\
\it        Kharkov, Ukraine & \it Charlottesville, VA &\it Charlotesville, VA\\
\it       chueshov@univer.kharkov.ua& \it il2v@virginia.edu & \it jtw3k@virginia.edu
\end{tabular}}

\begin{document}
\maketitle

\begin{abstract} {\noindent We consider the well-posedness of a model for a
 flow-structure interaction. This model describes the dynamics of
 an  elastic flexible plate  with clamped boundary conditions immersed in a supersonic flow. A perturbed wave equation describes the flow potential. The plate's out-of-plane displacement can be modeled by various nonlinear plate equations (including von Karman and Berger). We show that the linearized model is well-posed on the state space (as given by finite energy considerations) and generates a strongly continuous semigroup.
 We make use of these  results  to conclude global-in-time well-posedness for the fully nonlinear model.
\par
The proof of generation has two novel features, namely: (1) we
introduce a new flow potential velocity-type variable which makes it possible to cover
both subsonic and supersonic cases, and to
split the dynamics generating operator into a skew-adjoint component and a perturbation acting outside of the state space. Performing semigroup analysis also requires a nontrivial approximation of the domain of the generator. And (2) we make critical use of hidden regularity for
 the flow component of the model (in the abstract setup for the semigroup problem) which allows us run a fixed point argument and eventually conclude well-posedness. This well-posedness result for supersonic flows (in the absence of rotational inertia) has been hereto open. The use of semigroup methods to obtain well-posedness opens this model to long-time behavior considerations.  
\smallskip\par
\noindent {\bf Key terms:} flow-structure interaction, nonlinear plate,
 supersonic and subsonic flows,
 nonlinear semigroups, well-posedness, dynamical systems.
 \smallskip\par
\noindent {\bf MSC 2010:} 35L20, 74F10,   35Q74, 76J20
 
 }
\end{abstract}

\section{ Introduction}
\subsection{Physical Motivation}
The interaction of a thin, flexible structure with a surrounding flow of gas is one of the principal problems in aeroelasticity. Models of this type arise in many engineering applications such as studies of bridges and buildings in response to wind, snoring and sleep apnea in the human palate, and in the stability and control of wings and aircraft structures \cite{bal0,BA62,dowell1,HP02,Li03}.
 In general, for an abstract setup, we aim to model the oscillations of a thin flexible structure interacting with an inviscid potential flow in which it is immersed.  These models accommodate certain physical parameters, but one of the key parameters is the flow velocity of the unperturbed flow of gas.

Specifically, we deal with  a common flow-structure  PDE model which describes the interactive dynamics between a (nonlinear) plate and the surrounding potential flow (see, e.g., \cite{bolotin,dowellnon}). This model is one of the standard models in the applied mathematics  literature for the modeling of flow-structure
interactions (see, e.g., \cite{bolotin,dowellnon} and also \cite{dowell,dowell1} and the references therein).

The main goal of this paper is to present Hadamard well-posedness results  for the model in the presence of  \textit{supersonic flow velocities}.  While   {\it subsonic} flows
have received recent attention which has resulted in a  rather complete  mathematical theory of well-posedness  \cite{b-c,b-c-1,springer,jadea12,webster} and spectral behavior for reduced (linear) models \cite{bal3,shubov3}, this is not the case for  the
{\it supersonic flow velocities}\footnote{We exclude the models which are based
on the so-called ``piston" theory, see \cite[Chapter 4]{bolotin}, \cite[Part I]{dowellnon},
and also Remark~6.2.2 in \cite{springer} for a recent discussion.
}. The mathematical  difficulty in going from  a subsonic to supersonic regimes is apparent
when one inspects the {\it formal} energy balance. There is an apparent {\it loss of  ellipticity}  affecting  the static problem.  This, in turn, leads to  the appearance of boundary trace terms that  can not be handled  by
known (elliptic) PDE-trace theories.  Successful handling  of this issue yields  new methodology
which   is based on  appropriate (microlocal)  boundary trace  estimates and effectively
compensates for this  loss of ellipticity.
 The  method here presented  additionally covers (with minimal adjustments)  subsonic flows.

Thus, with respect to the supersonic model, this  paper addresses  the open question of well-posedness of {\it  finite energy solutions}, which is the most fundamental for future studies.
  Well-posedness results are necessary mathematically in order to begin long-time behavior and control studies of the model, which belong to the most interesting and pertinent mathematical studies in application for PDE models. Having shown well-posedness allows us to move into stability studies in the presence of control mechanisms \cite{cbms,redbook}.

\subsection{Notation}
For the remainder of the text we write $\xb$ for $(x,y,z) \in \realsthree_+$ or $(x,y) \in \Omega \subset \realstwo_{\{(x,y)\}}$, as dictated by context. Norms $||\cdot||$ are taken to be $L_2(D)$ for the domain dictated by context. Inner products in $L_2(\realsthree_+)$ are written $(\cdot,\cdot)$, while inner products in $L_2(\R^2\equiv\pd\R^3_+)$ are written $<\cdot,\cdot>$. Also, $ H^s(D)$ will denote the Sobolev space of order $s$, defined on a domain $D$, and $H^s_0(D)$ denotes the closure of $C_0^{\infty}(D)$ in the $H^s(D)$ norm
which we denote by $\|\cdot\|_{H^s(D)}$ or $\|\cdot\|_{s,D}$. We make use of the standard notation for the trace of functions defined on $\realsthree_+$, i.e. for $\phi \in H^1(\realsthree_+)$, $\gamma[\phi]=\phi \big|_{z=0}$ is the trace of $\phi$ on the plane $\{\xb:z=0\}$.

\subsection{PDE Description of the Model}
The model in consideration describes the interaction between a nonlinear plate with a field or flow of gas above it. To describe the behavior of the gas we make use of the theory of potential flows (see, e.g., \cite{bolotin,dowell} and the references therein) which produce a perturbed wave equation for the velocity potential of the flow. The oscillatory behavior of the plate is governed by the second order (in time) Kirchoff plate equation with a general nonlinearity.
  We will consider certain `physical' nonlinearities which are used in the modeling of the large oscillations of thin, flexible plates -  so-called \textit{large deflection theory}.
\par
The environment we consider is $\realsthree_+=\{(x,y,z): z \ge 0\}$. The plate
 is modeled by  a bounded domain $\Omega \subset \reals^2_{\{(x,y)\}}=\{(x,y,z): z = 0\}$ with smooth boundary $\partial \Omega = \Gamma$.
 The plate is  embedded  in  a `large' rigid body (producing the  so-called \textit{clamped} boundary conditions) immersed in an inviscid  flow (over body) with velocity $U \neq 1$ in the  negative $x$-direction\footnote{Here we normalize $U=1$ to be Mach 1, i.e. $0 \le U <1$ is subsonic and $U>1$ is supersonic.}. This situation corresponds to the dynamics of a panel element of an aircraft flying with the speed $U$, see, e.g., \cite{dowell1}.
\par
The scalar function $u: \Om\times \R_+ \to \reals$ represents the vertical displacement of the plate in the $z$-direction at the point $(x;y)$ at the moment $t$. We take the nonlinear Kirchoff type plate with clamped boundary conditions\footnote{While being the most physically relevant boundary conditions for the flow-plate model, clamped boundary conditions allow us to avoid certain technical issues in the consideration and streamline our exposition. Other possible and physically pertinent plate boundary conditions in this setup include: hinged, hinged dissipation, and combinations thereof \cite{lagnese}. }:

\begin{equation}\label{plate}\begin{cases}
u_{tt}+\Delta^2u+f(u)= p(\xb,t) & \text { in } \Om\times (0,T),\\
u(0)=u_0;~~u_t(0)=u_1,\\
u=\Dn u = 0 & \text{ on } \pd\Om\times (0,T).
\end{cases}
\end{equation}
The aerodynamical  pressure  $p(\xb,t)$  represents the coupling with the flow
and will be given below.
\par
In this paper we consider a general situation that
covers   typical nonlinear (cubic-type) force terms $f(u)$ resulting from  aeroelasticity  modeling \cite{bolotin,dowell,dowell1,HolMar78}.
These include:
\begin{assumption}\label{as}{\rm
\begin{enumerate}
  \item {\sl Kirchhoff model}: $u\mapsto f(u)$ is the Nemytski operator
with a function $f\in {\rm Lip_{loc}}(\R)$ which fulfills the  condition
\begin{equation}
    \label{phi_condition}
    \underset{|s|\to\infty}{\liminf}{\frac{f(s)}{s}}>-\la_1,
\end{equation}  where $\la_1$ is the first eigenvalue of the
biharmonic operator with homogeneous Dirichlet boundary conditions.
  \item {\sl Von Karman model:} $f(u)=-[u, v(u)+F_0]$, where $F_0$ is a given function
  from $H^4(\Om)$ and
the von Karman bracket $[u,v]$  is given by
\begin{equation*}
[u,v] = \partial ^{2}_{x} u\cdot \partial ^{2}_y v +
\partial ^{2}_y u\cdot \partial ^{2}_{x} v -
2\cdot \partial ^{2}_{xy} u\cdot \partial ^{2}_{xy}
v,
\end{equation*} and
the Airy stress function $v(u) $ solves the following  elliptic
problem
\begin{equation}\label{airy-1}
\Delta^2 v(u)+[u,u] =0 ~~{\rm in}~~  \Omega,\quad \Dn v(u) = v(u) =0 ~~{\rm on}~~  \pd\Om.
\end{equation}
Von  Karman equations are well known in nonlinear elasticity and
constitute a basic model describing nonlinear oscillations of a
plate accounting for  large displacements, see \cite{karman}
and also \cite{springer,ciarlet}  and
references therein.
  \item {\sl Berger Model:} In this case the feedback force $f$ has the form
  $$
  f(u)=\left[ \kappa \int_\Om |\nabla u|^2 dx-\Gamma\right] \Delta u,
$$
where $\kappa>0$ and $\Gamma\in\R$ are parameters, for some details and  references see \cite{berger0} and also
\cite[Chap.4]{Chueshov}.
\end{enumerate}
}
\end{assumption}
\par
For the flow component of the model, we make use of linearized
 potential theory, and we know \cite{BA62,bolotin,dowell1} that the (perturbed) flow potential $\phi:\realsthree_+ \rightarrow \reals$ must satisfy the perturbed wave equation below (note that when $U=0$ this is the standard wave equation):
\begin{equation}\label{flow}\begin{cases}
(\partial_t+U\partial_x)^2\phi=\Delta \phi & \text { in } \realsthree_+ \times (0,T),\\
\phi(0)=\phi_0;~~\phi_t(0)=\phi_1,\\
\Dn \phi = d(\xb,t)& \text{ on } \realstwo_{\{(x,y)\}} \times (0,T).
\end{cases}
\end{equation}
The strong coupling here takes place in the  downwash term of the flow potential (the Neumann boundary condition) by taking $$d(\xb,t)=-\big[(\partial_t+U\partial_x)u (\xb)\big]\cdot \mathbf{1}_{\Omega}(\xb)$$ and by taking the aerodynamical pressure
of the form
\begin{equation}\label{aero-dyn-pr}
p(\xb,t)=\big(\partial_t+U\partial_x\big)\gamma[\phi]
\end{equation}
 in (\ref{plate}) above. This gives the fully coupled model:
\begin{equation}\label{flowplate}\begin{cases}
u_{tt}+\Delta^2u+f(u)= \big(\partial_t+U\partial_x\big)\gamma[\phi] & \text { in } \Om\times (0,T),\\
u(0)=u_0;~~u_t(0)=u_1,\\
u=\Dn u = 0 & \text{ on } \pd\Om\times (0,T),\\
(\partial_t+U\partial_x)^2\phi=\Delta \phi & \text { in } \realsthree_+ \times (0,T),\\
\phi(0)=\phi_0;~~\phi_t(0)=\phi_1,\\
\Dn \phi = -\big[(\partial_t+U\partial_x)u (\xb)\big]\cdot \mathbf{1}_{\Omega}(\xb) & \text{ on } \realstwo_{\{(x,y)\}} \times (0,T).
\end{cases}
\end{equation}

\subsubsection{Parameters and New Challenges}
We do not include the full rotational inertia term $M_{\alpha}=(1-\al\Delta) u_{tt}$ in the LHS of plate equation, i.e., we take $\al =0$. In some considerations (see \cite{lagnese}) this term is taken to be proportional to the cube of the thickness of the plate, however, it is often neglected in large deflection theory. From the mathematical point of view, this term is regularizing in that it provides additional smoothness for the plate velocity $u_t$, i.e. $L_2(\Omega) \to H^1(\Omega)$. This is a key mathematical assumption in our analysis which separates it from previous supersonic considerations and increases the difficulty of the analysis. This will be further elaborated upon in the discussion of previous literature below. The case $\alpha > 0$  presents modeling difficulties in problems with flow coupling interface, but is often considered as a preliminary step in the  study limiting problems as $\alpha \searrow 0$
(see \cite{LBC96,b-c-1}, \cite{webster&lasiecka}, and also \cite{springer}) where subsonic regimes were studied.
On the other hand, the case when rotational terms are not included ($\alpha =0$)  leads to substantial new mathematical difficulties due to the presence of flow trace terms interacting with the plate.

The second key parameter is the unperturbed flow velocity $U$. Here we take $U\neq 1$
arbitrary. However, the supersonic case ($U>1$) is  the most interesting case from the point of view of application and engineering. Results in this case can be more challenging, due to the loss of strong ellipticity of the spatial flow operator in (\ref{flow}). For the subsonic case ($0\le U<1$) there are other methods available, see,
e.g., \cite{b-c,springer,jadea12,webster}. However, for  the \textit{non-rotational} case $\alpha =0 $ in the {\it supersonic regime}
$ U > 1 $, the problem of well-posedness of finite energy solutions is challenging and   {\it has been hereto open}.   As  is later expounded upon, the lack of  sufficient differentiability  and  compactness  for the plate velocity component $u_t$   renders the existing methods
(see \cite[Sections 6.5 and 6.6]{springer}, for instance)
 inapplicable.
 \par
 The aim of this paper is to provide an affirmative  answer to the well-posedness question in the (mathematically) most demanding case with $\alpha =0 $ and $ U > 1 $. In fact, we will show that the resulting dynamics generate a  {\it nonlinear semigroup associated with mild solutions}.
Though in  this treatment we focus  on  the most challenging case: $\alpha =0$ and $U > 1$, the new  methods developed apply to the full range $ U \ne 1$.
\subsection{Energies and State Space}
In the subsonic case $0 \le U <1$ energies can be derived by applying standard multipliers $u_t$ and $\phi_t$ along with boundary conditions to obtain the energy relations for the plate and the flow. This procedure leads to the energy which is bounded from below in the subsonic case.
However, it is apparent in the {\em supersonic} case that we will obtain an unbounded (from below) energy of the flow. Hence, we instead make use of the flow acceleration  multiplier $(\partial_t+U\partial_x)\phi \equiv \psi$. Our so-called change of variable is then $\phi_t \rightarrow (\phi_t+U\phi_x) = \psi$.
Thus for the flow dynamics, instead of $(\phi;\phi_t)$ we introduce
the phase variables  $(\phi;\psi)$.
\par
We then have a new  description of our  coupled system as follows:
\begin{equation}\label{sys1}\begin{cases}

(\partial_t+U\partial_x)\phi =\psi& \text{ in } \realsthree_+ \times(0,T),
\\
(\partial_t+U\partial_x)\psi=\Delta \phi&\text{ in } \realsthree_+\times (0,T),
\\
\Dn \phi = -\big[(\partial_t+U\partial_x)u (\xb)\big]\cdot \mathbf{1}_{\Omega}(\xb) & \text{ on }
 \realstwo_{\{(x,y)\}} \times (0,T),
\\  u_{tt}+\Delta^2u+f(u)=\gamma[\psi]&\text{ in } \Om\times (0,T),\\
u=\Dn u = 0 & \text{ on } \pd\Om\times (0,T).\\
\end{cases}
\end{equation}
This leads to the following (formal) energies, arrived at via Green's Theorem:
\begin{align}\label{energies}
E_{pl}(t) =& \dfrac{1}{2}\big(||u_t||^2+||\Delta u||^2\big) +\Pi(u),\\
E_{fl}(t) = & \dfrac{1}{2}\big(||\psi||^2+||\nabla \phi||^2\big),
\notag
\\
\cE(t) = & E_{pl}(t)+E_{fl}(t),\notag
\end{align}
where $\Pi(u)$ is a potential of the nonlinear force $f(u)$, i.e. we assume that
$f(u)$ is a Fr\'echet derivative of $\Pi(u)$,
$f(u)=\Pi'(u)$. Hypotheses concerning $\Pi(u)$ are motivated by
the examples described in Assumption~\ref{as}  and will be given later (see the statement
of Theorem~\ref{abstract}).
\par
With these energies, we have the formal energy relation\footnote{ For some details
in the rotational inertia case we refer to \cite{b-c-1}, see also the proof
of relation (6.6.4) in
\cite[Section 6.6]{springer}.
}

\begin{equation}\label{energyrelation}
\cE(t)+ U\int_0^t <u_x,\gamma[\psi]> dt = \cE(0). \end{equation}
This energy relation provides the first motivation for viewing the dynamics (under our change of phase variable) as comprised of a generating piece and a perturbation.
\par
Finite energy constraints manifest themselves in the natural requirements on the functions $\phi$ and $u$:
\begin{equation}\label{flowreq}
\phi(\xb,t) \in C(0,T; H^1(\realsthree_+))\cap C^1(0,T;L_2(\realsthree_+)), \end{equation} \begin{equation}\label{platereq}
u(\xb,t) \in C(0,T; H_0^2(\Omega))\cap C^1(0,T;L_2(\Omega)).\end{equation}
In working with well-posedness considerations (and thus dynamical systems), the above finite energy constraints lead to the  so-called finite energy space, which we will take as our state space: \begin{equation}\label{space-Y}
Y = Y_{fl} \times Y_{pl} \equiv \big(H^1(\realsthree_+)\times L_2(\realsthree_+)\big) \times \big(H_0^2(\Omega) \times L_2(\Omega)\big).
\end{equation}
\begin{remark}
The energy $E_{fl}(t)$ defined  above  coincides with
 the  energy $E^{(2)}_{fl}(t)$ introduced in \cite{LBC96,b-c-1}, see also \cite[Sect.6.6]{springer}. As previously indicated, in the subsonic case, the standard flow multiplier $\phi_t$ is used in the analysis, rather than $\psi = \phi_t + U\phi_x$; this produces differing flow energies and an interactive  term  $E_{int}(t)$, which does not appear here (see, e.g., \cite{b-c,springer,jadea12,webster}).  More specifically,  the flow component of the energy in this case is given by
 $$E^{(1)}_{fl}(t) \equiv \frac{1}{2} \left( ||\phi_t||^2 + ||\nabla \phi||^2 - U^2 ||\partial_x \phi||^2 \right) $$
 and the  interactive energy $E_{int} = U <\gamma[\phi], \partial_x u > $.
 The total energy  defined as a sum of the three components $\cE(t) =E^{(1)}_{fl}(t) + E_{pl}(t) + E_{int}(t)  $ satisfies
 $\cE(t) = \cE(s) $.
 We note, that in the {\it supersonic}   case   the flow part of the energy $E^{(1)}_{fl}(t) $ is no longer nonnegative.
 This, being  the source of major mathematical difficulties, necessitates  a   different approach.
 In fact, the new representation of the energies as in (\ref{energies}) provides good topological measure for the sought after solution, however the energy balance is {\it lost} in (\ref{energyrelation}) and, in addition, the boundary term is
 ``leaking energy" and  involves the traces of $ L_2$ solutions, which are possibly {\it not defined}  at all.

  In view of the above,  our  strategy will be based on  (i)  developing    theory for the traces of the flow solutions; (ii)  counteracting  the loss of energy balance relation. The first task will be accomplished by exploiting
      {\it sharp}
  trace regularity  in hyperbolic Neumann  solutions  (see \cite{l-t-sharp,miyatake1973,sakamoto,tataru} for related results). The second task will benefit critically from the presence of the  nonlinearity.
\end{remark}

\subsection{Definitions of Solutions}\label{solutions}
In the discussion below, we will encounter strong (classical), generalized (mild), and weak (variational) solutions.
In our analysis we will be making use of semigroup theory, hence we will work with \textit{generalized} solutions; these are strong limits of strong solutions. These solutions  satisfy an integral formulation of (\ref{flowplate}), and are called \textit{mild} by some authors. In our treatment, we will produce a unique generalized solution, and this, in turn, produces a unique weak solution, see, e.g., \cite[Section 6.5.5]{springer} and \cite{webster}.

We now define strong and generalized  solutions:
\vskip.25cm
\noindent\textbf{Strong Solutions}.
A pair of functions $\big(\phi(x,y,z;t);u(x,y;t)\big)$
satisfying
\eqref{flowreq}
   and \eqref{platereq}  is said to be a strong solution to (\ref{flowplate}) on $[0,T]$ if
\begin{itemize}
\item $(\phi_t;u_t) \in L^1(a,b; H^1(\realsthree_+)\times H_0^2(\Omega))$ and $(\phi_{tt};u_{tt}) \in L^1(a,b; L_2(\realsthree_+)\times L_2(\Omega))$
    for any $[a,b] \subset (0,T)$.
\item  $\Delta^2 u(t)  - U  \gamma[\Dx \phi(t) ]\in
L_2(\Omega)$ (thus $u(t) \in H^{7/2}(\Omega)\cap H_0^2(\Om)$) and
the equation $ u_{tt}+\Delta^2 u +f(u)=p(\xb,t)$ holds in $H^{-1/2}(\Omega)$ for $ t \in(0,T)$ with $p$ given by \eqref{aero-dyn-pr}.
\item  $ (U^2-1) \partial_x^2 \phi (t) - (\partial_y^2 + \partial_z^2) \phi(t)  \in
L_2(\R^3_{+}) $  with boundary conditions $\Dn \phi(t) \in H^1(\R^2) $ for all $ t \in(0,T)$ and
satisfying the relation
$\Dn \phi = -\big[(\partial_t+U\partial_x)u (\xb)\big]\cdot \mathbf{1}_{\Omega}(\xb)$  on
$\realstwo \times (0,T)$.
Moreover
 $(\partial_t + U\partial_x)^2\phi=\Delta \phi$ holds  for almost all  $ t \in(0,T)$ and $(x,y,z) \in \realsthree_+$.
\item The initial conditions are satisfied: $\phi(0)=\phi_0, ~~\phi_t(0)=\phi_1, ~~u(0)=u_0, ~~u_t(0)=u_1.$
\end{itemize}
\begin{remark}\label{se-str-sol}
The smoothness properties in the definition above are motivated by
the description of the generator of the linear problem
in the {\em supersonic} case $U>1$
 which is given  below, see relation \eqref{dom-bA} and Lemma~\ref{domain}. In the subsonic case regular solutions display more regularity  (see, e.g., \cite[Secions 6.4 and 6.5]{springer} and \cite{jadea12}).
The above analysis also reveals that  the degraded  differentiability of strong solutions is due to the  the loss of elipticity in the supersonic regime  and non-Lopatinski character of the boundary conditions.
\end{remark}

As stated above, generalized solutions are strong limits of strong solutions; these solutions  correspond to semigroup solutions for an initial datum outside of the domain of the generator.\vskip.25cm
\noindent\textbf{Generalized Solutions}.
A pair of functions $\big(\phi(x,y,z;t);u(x,y;t)\big)$ is said to be a generalized solution to  problem (\ref{flowplate}) on the interval $[0,T]$ if (\ref{platereq}) and (\ref{flowreq}) are satisfied and there exists a sequence of strong solutions $(\phi_n(t);u_n(t))$ with some initial data $(\phi^n_0,\phi^n_1; u^n_0; u^n_1)$ such that
$$\lim_{n\to \infty} \max_{t \in [0,T]} \Big\{||\partial_t\phi-\partial_t \phi_n(t)||_{L_2(\realsthree_+)}+||\phi(t)-\phi_n(t)||_{H^1(\realsthree_+)}\Big\}=0$$ and
$$\lim_{n \to \infty} \max_{t \in [0,T]} \Big\{||\partial_t u(t)-\partial_t u_n(t)||_{L_2(\Omega)} + ||u(t) - u_n(t)||_{H_0^2(\Omega)}\Big\}=0.$$

We can show that in the case when the nonlinear term  $f$ is locally Lipschitz from $H^2_0(\Omega)$ into $L_2(\Om)$
 the generalized solutions are in fact {\em weak solutions}, i.e.,
they satisfy the corresponding variational forms
(see Definition 6.4.3 in \cite[Chapter 6]{springer}).
 This can be  verified for strong solutions by straightforward integration with the use of regularity exhibited by strong solutions. Using the (strong) limit definition of generalized solutions,
  we can pass in the limit
  and show that the generalized solution satisfies the weak formulation of (\ref{flowplate}). This weak solution is in fact unique - for the proof, we defer to the method presented in \cite[Chapter 6]{springer}.

\subsection{Description  of Past Results}

Flow-structure models have  attracted considerable attention in the past mathematical literature, see, e.g., \cite{bal4,bal3,b-c, LBC96,b-c-1,chuey,springer,jadea12,dowellnon,webster&lasiecka,ACC,ryz,ryz2,shubov3,shubov1,webster}
and the references therein.
However, the vast majority of the work done have been devoted to numerical and experimental studies, see,
e..g.,
\cite{bal0,BA62,bolotin,dowell,dowell1,dowellnon,HP02} and also the survey~\cite{Li03}
and the literature cited there.  Much of the studies has been based on linear one-dimensional-special geometries plate models  where  the  goal was to determine the speed at which flutter occurs, see
\cite{bal0,BA62,bolotin,dowell1,HP02,Li03} for instance.
 More recently the study of linear models with a one dimensional structure (beam)
and Kutta-Jukovsky  boundary conditions  found renewed interest and
have been extensively  pursued in \cite{bal4,shubov3,shubov2,shubov1}.  This line of work has focused on spectral properties of the system, with particular emphasis on identifying aeroelastic eigenmodes corresponding to  the associated  Possio integral  equation.
\par
In contrast, our interest here concerns PDE aspects of the problem, including the fundamental issue of well-posedness
of {\it  finite energy}  solutions corresponding to nonlinear flow-plate interaction  in the principal case for the parameters $\alpha$ and $U$ with clamped plate boundary conditions.
\par
In all parameter cases, one is faced with low regularity of boundary traces
due to the failure of Lopatinski conditions (unlike the Dirichlet case \cite{sakamoto}, where there is no loss of regularity to  wave  solutions in their boundary traces).
In fact, the first contribution to the mathematical analysis of the problem  is \cite{LBC96,b-c-1} (see also \cite[Section 6.6]{springer}), where the case  $\alpha > 0 $ is fully treated.  The method employed in \cite{LBC96,b-c-1,springer}
relies on the following main ingredients:
(1) sharp microlocal estimates for the solution to the wave equation driven by $H^{1/2}(\Omega) $ Neumann  boundary data given by $u_t + U u_x$. This gives $\phi_t|_{\Omega} \in
L_2(0,T; H^{-1/2}(\Omega)$  \cite{miyatake1973} (in fact more regularity is presently known:
$H^{-1/3}(\Omega) $ \cite{l-t-sharp,tataru});  and (2) the regularizing effects on the velocity $u_t$ (i.e. $u_t \in H^1(\Omega)$) rather than just $L_2(\Omega)$), when $\alpha >0$.
 The above ingredients, along with an explicit solver for the 3-dimensional wave equation and a Galerkin approximation for the structure  allows one to construct   a fixed point for the appropriate
 solution map. The method  works well in both cases $0<U<1$ and  $U > 1 $.  Similar ideas were used  more recently \cite{ryz,ryz2} in  the case when thermoelastic effects are added to the model; in this case the dynamics also exhibit
$H^1(\Omega)$ regularity of the velocity in both the rotational and non-rotational cases due to the analytic regularizing  effects induced by thermoelasticity \cite{redbook}. However, when
 $\alpha  =0$, and thermoelastic smoothing is not accounted for, there is no additional regularity of $u_t$ beyond $L_2(\Omega)$.  In that case the corresponding estimates
 become singular. This destroys the applicability of previous methods. In summary, much of the work on this problem to date has assumed  the presence of additional regularizing terms in the plate equation, or depends critically on the condition $U <1$. A recent book (\cite[Chapter 6]{springer}) provides an account of relevant results, including more recent applications of the compactness method in the case $\alpha=0$ and $0 \le U <1$. Existence of a nonlinear semigroup capturing finite energy solutions   has been shown in \cite{webster}, see also \cite{jadea12}.
\par
 In this treatment we take $\alpha=0$, corresponding to the more difficult non-rotational model, and we approach the problem with $0\le U\neq 1$ from the semigroup point of view - without any reliance on explicit solvers for the flow equation or Galerkin constructions. The advantage of this approach, in addition to solving the fundamental well-posedness question, is the potential for an array of important generalizations, including more general flow equations and more general nonlinearities appearing in the structure.
 Moreover, it may be possible to view the supersonic global solution we arrive at for $\alpha=0$ as the uniform limit of solutions as $\alpha \downarrow 0$ (as in the subsonic case \cite{webster&lasiecka}).

The  main  mathematical difficulty of the problem under  consideration
is  the presence of  the   boundary  terms:  $ (\phi_t +
 U \phi_x)|_{\Omega} $  acting as the aerodynamical pressure  in the model. When $U =0$, the corresponding  boundary terms exhibit monotone behavior with respect to   the energy  inner product
 (see \cite{bucciCL07}, \cite[Section 6.2]{springer} and \cite{cbms})
 which is topologically equivalent  to  the topology of the  space $Y$ given by \eqref{space-Y}.   The latter  enables  the use of monotone operator theory (\cite{bucciCL07}, \cite[Section 6.2]{springer} and \cite{cbms}).  However, when
  $U > 0$ this is no longer true   with respect to the  topology induced by the energy spaces.  The   lack of the natural dissipativity for both
interface traces, as well as the nonlinear term in the  plate equation,
make the  task of proving well-posedness  challenging. In the subsonic case, semigroup methods were applied to the problem by implementing certain bounded adjustments to the inner product structure of the state space which then produced shifted dissipativity \cite{webster}. This type of consideration is not possible here, owing to the degeneracy of the standard
 energy functional when $U \ge 1$.
\par
In contrast to these works, our method and results
{\it  do not depend on any smoothing mechanism} (as we take $\alpha =0$) and are applicable {\em for all} $U \ne 1$.
The key ingredients rely on the development of a suitable trace theory for the velocity of the flow and  implementing the corresponding  estimates with semigroup theory in extended spaces.
In this way obtained a-priori estimates allow for a construction of a nonlinear semigroup which evolves finite energy solutions.

\subsection{Statement of Main Results}
Recall, our state space for the analysis to follow is
$$
Y \equiv H^1(\R_+^3) \times L_2(\R_+^3) \times H_0^2(\Omega) \times L_2(\Omega).
$$
\begin{theorem}[{\bf Linear}]\label{th:lin}
Consider linear problem in (\ref{flowplate}) with $f(u) =0$.  Let $ T > 0 $. Then,   for every initial datum
$( \phi_0, \phi_1; u_0, u_1 ) \in Y $ there exists unique generalized solution
\begin{equation}\label{phi-u0reg}
(\phi (t), \phi_t(t); u(t), u_t(t)) \in C([0, T ],  Y)
\end{equation}
 which   defines  a $C_0$-semigroup
$T_t : Y \rightarrow Y $ associated with  (\ref{sys1}) (where $ f =0$).
\par
 For any initial data in
  \begin{equation}\label{Y1}
Y_1 \equiv \left\{ y= (\phi,\phi_1;u,u_1)  \in Y\;   \left| \begin{array}{l}
\phi_1 \in H^1(\R^3_{+} ),~~ u_1 \in H_0^2(\Omega),  ~\\ -U^2 \Dx^2 \phi  + \Delta \phi \in L_2(\R^3_{+}), \\
\Dn  \phi = - [u_1 +U \Dx u ]\cdot {\bf 1}_{\Omega} \in H^1(\R^2), \\
~
-\Delta^2 u + U\gamma[ \Dx \phi] \in L_2(\Omega) \end{array} \right. \right\}
\end{equation}
   the corresponding  solution is also strong.
\end{theorem}

\begin{theorem}[{\bf Nonlinear}]\label{th:nonlin}
Let $ T > 0 $ and let $f(u)$ be any nonlinear internal force (Kirchhoff, von Karman, or Berger) given by Assumption \ref{as}. Then,   for every initial data
$( \phi_0, \phi_1; u_0, u_1 ) \in Y $ there exists unique generalized solution
$(\phi, \phi_t; u, u_t)$ to (\ref{flowplate})
 possessing property \eqref{phi-u0reg}. This solution is also weak and generates
  a  nonlinear  continuous  semigroup
$S_t : Y \rightarrow Y$ associated with   (\ref{sys1}).
\par
If $( \phi_0, \phi_1; u_0, u_1 ) \in Y_1$, where
 $Y_1 \subset Y$ is given by \eqref{Y1}, then the corresponding solution is also strong.
\end{theorem}

  \begin{remark}
  In comparing the results obtained with a subsonic case, there are two major differences at the qualitative level:
\begin{enumerate}
  \item  Regularity of strong solutions obtained in the subsonic case \cite{jadea12} coincides with regularity expected for classical solutions. In the supersonic case, there is a  loss of differentiability in the flow  in the tangential  $x$ direction, which then propagates to the  loss of differentiability in the structural variable $u$.
      \item
 In the subsonic case one shows that the resulting semigroup is {\it bounded} in time,
  see \cite[Proposition 6.5.7]{springer} and also \cite{jadea12,webster}.
  This property  could not be shown in this analysis, and most likely does not hold.  The leack of the energy in energy relation can not be compensated for by the nonlinear terms (unlike the subsonic case).
\end{enumerate}
   \end{remark}
\subsection{Proof Strategy}
In order to orient and guide the reader through various stages   of the proof, we  briefly outline the main
steps.
\begin{enumerate}
\item As motivated by the linear theory in the subsonic case, we use the modified energy (as given in the previous section) to setup the linear problem abstractly. We decompose the linear dynamics into a dissipative piece $\bA$ (unboxed below) and a perturbation piece $\bP$ (boxed below): \begin{equation}\label{sys1*}\begin{cases}
(\partial_t+U\partial_x)\phi =\psi& \text{ in } \realsthree_+ \times(0,T),
\\
(\partial_t+U\partial_x)\psi=\Delta \phi-\mu \phi&\text{ in } \realsthree_+\times (0,T),
\\
\Dn \phi = -\partial_tu\cdot \mathbf{1}_{\Omega}(\xb) -\framebox{$U\partial_xu \cdot \mathbf{1}_{\Omega}(\xb) $}& \text{ on }
\realstwo_{\{(x,y)\}} \times (0,T),
\\
u_{tt}+\Delta^2u=\gamma[\psi]&\text{ in }  \Om \times (0,T),\\
u=\Dn u = 0 & \text{ in } \pd\Om \times (0,T).\\
\end{cases}
\end{equation} We then proceed to show that $\bA$ (corresponding to the unboxed dynamics above) is $m$-dissipative on the state space.  While ``dissipativity" is natural and built in  within the structure of the problem, the difficulty encountered is in establishing   the {\it maximality} property for the generator.
The  analysis of the   the resolvent operator  is no longer reducible to  strong elliptic theory (unlike the classical wave equation). The ``loss of ellipticity" prevents us from using  the known tools.    To handle this   we develop  a non-trivial approximation argument to justify the formal calculus; maximality will then be achieved  by constructing a suitable
bilinear form to which  a version of  Lax-Milgram argument applies.
\item To handle the ``perturbation" of the dynamics, $\bP$ (boxed) we cast the problem into an abstract boundary control framework.
In order to achieve this, a critical ingredient in the proof is demonstrating  ``hidden"  boundary regularity for the acceleration potential $\psi$ of the flow. It will be shown that this component is an element of a negative Sobolev space
$L_2(0, T; H^{-1/2} (\Omega))$.
The above regularity allows us to show that the term $<u_x,\gamma[\psi]>$ is well-defined via duality.
Consequently,   the  problem  with  the  ``perturbation" of the dynamics $\bP$,   can be recast
as an abstract boundary control problem with appropriate continuity properties of the control-to-state maps.

\item Then, to piece the operators together as $\bA+\bP$, we make use of variation of parameters with respect to the generation property of $\bA$ and appropriate dual setting. This yields an integral equation on the state space (interpreted via duality) which must be formally justified in our abstract framework. This is critically dependent upon point (2.) above. We then run a fixed point argument on the appropriate space to achieve a local-time solution for the fully linear Cauchy problem  representing formally  the evolution $y_t=(\bA+\bP)y \in Y$. In order to identify its generator, we apply Ball's theorem \cite{ball} which then yields global solutions.
\item Lastly, to move to the nonlinear problem, we follow the standard track of writing the nonlinearity as a locally Lipschitz perturbation on the state space $Y$.
 In the most difficult case of the von Karman nonlinearity
    the latter is possible due to the established earlier  ``sharp" regularity of Airy's stress function \cite[pp.44-45]{springer}. This allows us to
 implement  local theory with a priori bounds (for $T$ fixed) on the solution.
The  global a priori bounds  result from an appropriate compactness-uniqueness argument supported by a maximum principle applied to Monge-Ampere equation; the latter implies
 $||u||_{L_{\infty} (\Omega) } \leq C ||[u,u]||^{1/2}_{L_1} $ for $ u \in  H^2(\Omega) \cap H_0^1(\Omega)$, see
\cite[Sections 1.4 and 1.5, pp. 38-58]{springer} for details.
  The above procedure  yields a nonlinear semigroup which, unlike the case of subsonic flow,  is not necessarily bounded in time (see the case of subsonic dynamics \cite{springer,jadea12,webster}) and the resolvent of this semigroup is not compact.
\end{enumerate}
\section{Abstract Setup}
\subsection{Operators and Spaces}

Define the operator $A=-\Delta+\mu$  with some
 $\mu>0$\footnote{
We include the term $\mu I$ in the operator $A$ to avoid a zero point in the spectrum. After we produce our semigroup analysis, we will negate this term in the abstract formulation of the problem in order to maintain the equivalence of the abstract problem and the problem as given in (\ref{sys1}).
}
 and with the domain
\[
\cD(A)=\{u\in H^2(\realsthree_+): ~\Dn u = 0\}.
 \]
 Then $\cD(A^{1/2})=H^1(\realsthree_+)$ (in the sense of topological equivalence).
We also
 introduce the standard linear plate operator with clamped boundary conditions:
 $\mathscr{A}=\Delta^2$ with the domain
 \[
 \cD(\mathscr{A})=\{u\in H^4(\Omega):~u|_{\partial \Omega}=\Dn u|_{\partial \Omega}=0\}=\big(H^4\cap H_0^2\big)(\Omega).
  \]
  Additionally, $\mathscr{D}(\mathscr{A}^{1/2})=\big(H^2\cap H_0^1\big)(\Omega)$. Take our state variable to be $$y\equiv(\phi, \psi; u, v) \in \big(\cD(A^{1/2})\times L_2(\realsthree_+)\big)\times\big( \cD(\cA^{1/2})\times L_2(\Omega)\big)\equiv Y.$$ We work with $\psi$ as an independent state variable, i.e., we are not explicitly taking $\psi = \phi_t+U\phi_x$ here.

To build our abstract model, let us define the operator $\bA: \cD(\bA) \subset Y \to Y$ by
\begin{equation}\label{op-A}
\bA \begin{pmatrix}
\phi\\
\psi\\ u\\ v
\end{pmatrix}=\begin{pmatrix}-U\partial_x \phi+\psi\\- U\partial_x\psi-A(\phi+Nv)\\ v \\ -\cA u+N^*A\psi \end{pmatrix}
\end{equation}
where the Neumann map $N$ is defined as follows:
$$
Nf=g~\iff~(-\Delta+\mu)g=0~\text{ in }~\R^3_+  ~ \text{ and } ~ \Dn g = f ~\text{ for } ~ z=0.
$$
Properties of this map on bounded domains and $\realsthree_+$ are well-known
(see, e.g., \cite[p.195]{redbook} and \cite[Chapter 6]{springer}), including the facts that
$N\, : H^s(\R^2)\mapsto H^{s+1/2}(\R^3_+)$ and
\begin{equation*}
N^*Af=\gamma[f]~~\mbox{for}~~f \in \cD(A),
\end{equation*}
and via density, this formula holds for all $f \in \cD(A^{1/2})$ as well. Additionally, when we write $Nv$ for $v: \Omega \to \reals$, we implicitly mean $N v_{ext}$, where $v_{ext}$ is the extension\footnote{We must utilize
this zero extension owing to the structure of the boundary condition
for $\Dn\phi$ in \eqref{sys1*}.}
 by $0$ outside $\Omega$.
\par
The domain of $\cD(\bA) $ is given by
\begin{equation}\label{dom-bA}
\cD(\bA) \equiv \left\{ y =  \begin{pmatrix}
\phi\\
\psi\\ u\\ v
\end{pmatrix}\in Y\; \left| \begin{array}{l}
-U \Dx \phi + \psi \in H^1(\R^3_+),~\\ -U \Dx \psi  -A (\phi + N v ) \in L_2(\R^3_+), \\
v \in \cD(\cA^{1/2} )= H_0^2(\Omega),~
-\cA u + N^* A \psi \in L_2(\Omega) \end{array} \right. \right\}
\end{equation}
We can further characterize the domain:
\par
Since on  $\cD(\bA)$ we have
that $\psi=U \Dx \phi + h$ for some  $h\in H^1(\R^3_+)$,
 then
 $\gamma [\psi] \in H^{-1/2}(\R^2)$
 (we identify $\R^2$ and $\pd\R^3_+$).
 This implies   $\gamma [\psi]\big|_{\Om} \in H^{-1/2}(\Om) =[H_0^{1/2}(\Om)]'$.
 Therefore we have that
  $\cA u  \in H^{-1/2}(\Omega) \subset
[\cD(\cA^{1/8})]'$  (recall that by interpolation
the relation $\cD(\cA^{1/8}) \subset H_0^{1/2}(\Omega)$ holds).  Thus
$$
u \in \cD(\cA^{7/8}) \subset H^{7/2}(\Omega).
$$
Moreover, for smooth functions $\widetilde{\psi} \in L_2(\R^3_+) $
we have that
$$
(U \Dx \psi + A \phi,\widetilde{ \psi})_{L_2(\R^3_+)}=
(U \Dx \psi + A (\phi  +  N v),\widetilde{ \psi})_{L_2(\R^3_+)}-
< v, \gamma [\widetilde{\psi}] >_{L_2(\R^2)}.
 $$
Thus  on the account that $(\phi,\psi;u,v) \in \cD(\bA)$, so that
$U \Dx \psi  -A (\phi + N v ) \in L_2(\R^3_+)$ and $v \in H_0^2(\Omega)$
we have that
$$
\big|(U \Dx \psi + A \phi,\widetilde{\psi})_{\R^3_+}\big|\le C
\|\widetilde{\psi}\|_{\R^3_+}+
\|\gamma [\widetilde{\psi}]\|_{H^{-2}(\R^2)} $$
for any $\widetilde{\psi} \in L_2(\R^3_+) $ with $\gamma [\widetilde{\psi}] \in H^{-2}(\R^{2})$.
\par
Writing $ \Delta  \phi =(- U \Dx \psi)  + l_2 $ for some $l_2\in L_2(\R_+^3) $
with the boundary conditions $\Dn \phi = v $, where $v\in H^2_0(\Om)$, we easily conclude from standard elliptic theory that
\begin{equation}\label{phi-p}
    \phi = - U A^{-1} \Dx \psi + h_2 ~~\mbox{for some $h_2 \in H^2(\R^3_+)$}.
\end{equation}
Substituting  this relation into  the first condition characterizing the domain
in \eqref{dom-bA}
 we obtain
$$ U^2 \Dx A^{-1} \Dx \psi + \psi = h_1 \in H^1(\R^3_+) $$
which implies
$$
 U^2 \Dx^2 A^{-1} \Dx \psi + \Dx \psi  \in L_2 (\R^3_+)
$$
Introducing the variable $p \equiv A^{-1} \Dx \psi$ one can see
that $p$ satisfies wave equation in the supersonic case
\begin{equation}\label{wave}
 (U^2 -1) \Dx^2 p +(- \Delta_{y,z}+\mu) p  \in L_2(\R^3_+ ),
 \end{equation}
where $ \Dn p =0 $ on the boundary $z =0 $ distributionally.
\par
The observations above lead to the following description of the  domain $\cD(\bA)$.

\begin{lemma}\label{domain}
The domain of $\bA$, $\cD(\bA) \subset Y$, is characterized by:\
$y \in \cD(\bA)$ implies
\begin{itemize}
\item
$    y =(\phi, \psi, u, v)  \in Y, ~~\gamma [\psi] \in H^{-1/2}(\Omega)$,
 \item
$-U \Dx \phi + \psi \in H^1(\R^3_+)$,
\item
$v \in \cD(\cA^{1/2} ) = H_0^2(\Om),~~
 u\in \cD(\cA^{7/8})$,
\item
$ |(-U \Dx \psi - A \phi ,\widehat{ \psi})_{\R^3_+}| < \infty,~~~ \forall  ~\widehat{\psi}\in L_2(\R^3_+) ~\text{ with }~ \gamma [\widehat{\psi}]\in H^{-2}(\R^{2})$,
\item
$ U^2 \Dx^2 A^{-1} \Dx \psi + \Dx \psi  \in L_2 (\R^3_+) $ or (\ref{wave}) holds.
Since by \eqref{phi-p} $ \phi = - U p + h_2 $ for some $h_2 \in H^2(\R^3_+)$,
equation (\ref{wave}) can be also written explicitly in terms of $\phi $ as
\begin{equation*}
(U^2 -1) \Dx^2 \phi+( - \Delta_{y,z}+\mu) \phi  \in L_2(\R^3_+)
\end{equation*}
where $ \Dn \phi  =v_{ext}  $ on the boundary $z =0 $.
\end{itemize}
\end{lemma}
\par

\subsection{Cauchy Problem and Unbounded Perturbation in Extended Space}
With this setup, we  will be in a position to show that $\bA$ is $m$-dissipative. A peculiar feature introduced by the presence of the supersonic parameter is  the  loss of uniform ellipticity  in  the static version of the perturbed wave operator
and the loss of compactness in the resolvent operator. The domain of  $\bA$ does not posses sufficient regularity. To cope with this difficulty, suitable approximation of  the domain will be introduced.
In view of this, the proof of the maximality property is involved here. The obtained result  will give that the Cauchy problem
\begin{equation}\label{eq-trankt-0}
y_t=\bA y, ~y(0)=y_0 \in Y
\end{equation}
 is well-posed on $Y$. We will then consider the (semigroup) perturbation \begin{equation*}
\bP\begin{pmatrix}\phi\\\psi\\u\\v \end{pmatrix}=\begin{pmatrix}0\\-UAN\partial_x u\\0\\0\end{pmatrix}.
\end{equation*}
The issue here is the  unbounded ``perturbation",  which  does not reside in the state space $Y$. Note that
${\mathscr{R}}\{AN\} \not\subset L_2(\R^3_{+})$, and
only the trivial element $0$ is in the domain of $AN$ when the latter  considered with the values in $Y$.
This fact  forces  us to construct a perturbation theory which operates in extended  (dual) spaces.
This step will rely critically on ``hidden" boundary regularity of the acceleration potential $\psi$
- established in the next section.   As a consequence, we  show that
 the resulting Cauchy problem $y_t = (\bA + \bP)y, ~y(0)=y_0 \in Y$  yields well-posedness for the full system in (\ref{sys1*}). Application of Ball's theorem~\cite{ball}  allows us  to conclude that  $\bA + \bP$, with  an appropriately defined domain, is a generator of a strongly continuous semigroup on $Y$.
\section{Proof of Main Result}
\subsection{ Hidden Regularity of $\gamma[\psi]$}
We consider  the following initial boundary value problem:
\begin{equation}\label{flow-h-U}\begin{cases}
(\partial_t+U\partial_x)^2\phi=\Delta \phi & \text { in } \realsthree_+,\\
\phi(0)=\phi_0;~~\phi_t(0)=\phi_1, \\
\Dn \phi = h(\xb,t) & \text{ on } \realstwo_{\{(x,y)\}}.
\end{cases}
\end{equation}
We assume
\begin{equation}\label{h-cond3.1}
 h(\xb,t)\in L_2^{loc}(0,T;  L_2(\R^2)).
\end{equation}
We note  that
initially, in our studies of the partial dynamics associated to the dissipative part of the dynamics (the system in (\ref{sys1*}) with the boxed term removed), we will take $h(\xb,t) = -\partial_t u$,
where $u$ is a plate component of a 
generalized solution to the unboxed part
of \eqref{sys1*}.
 Later, in considering the perturbation of the dissipative dynamics, we will take
\[
h(\xb,t)=-\big[\partial_t u+U\partial_x \overline u\big]_{ext},
 \]
where $\overline u$ is some other function which belongs the same smoothness class as $u$.
Therefore the requirement in \eqref{h-cond3.1} is reasonable.
 \par
Let $\phi$ be the energy type solution of \eqref{flow-h-U},
i.e.
\begin{equation}\label{w-eq-U-def}
    (\phi,\phi_t)\in L_2(0,T; H^1(\R^3_+)\times L_2(\R^3_+)),~~~\forall\, T>0.
\end{equation}
These solutions exists, at least for sufficiently smooth $h$ (see \cite{miyatake1973} and \cite{sakamoto}).

Our goal is to estimate   the  trace of the acceleration potential  $\phi_t + U \phi_x $ on $z=0$.
The  a priori regularity of $\phi(t) \in H^1(\R^3_+) $ implies via trace theory
 $\phi_x (t)|_{z=0} \in H^{-1/2}(\R^2) $. However,  the a priori regularity of $\phi_t $  does not allow to infer, via trace theory, any notion of a trace.
 Fortunately  we will be able to show that this trace does exist as a distribution and can be measured in a negative Sobolev space. The corresponding result reads:

  \begin{lemma}[{\bf Flow Trace Regularity}]\label{le:FTR}
  Let \eqref{h-cond3.1} be in force.
If $\phi(\xb,t)$  satisfies \eqref{flow-h-U} and \eqref{w-eq-U-def}, then $$\partial_t\gamma[\phi],~~ \partial_x\gamma[\phi]  \in L_2(0,T;H^{-1/2}(\R^2))~~~~\forall\, T>0.
$$
Moreover,  with $\psi = \phi_t + U \phi_x $ we have
\begin{equation}\label{trace-reg-est-M}
\int_0^T\|\gamma[\psi](t)\|^2_{H^{-1/2} (\R^2)}dt\le C_T\left(
E_{fl}(0)+
 \int_0^T\| \Dn \phi(t) \|^2dt\right)
\end{equation}
\end{lemma}
The above result is critical for the arguments in later portion of this treatment. Specifically, the above result holds for \textit{any} flow solver; we will be applying this result in the case where $\Dn \phi = - v \in C(0,T;L_2(\Omega))$ coming from a semigroup solution generated by $\bA$.
\begin{proof}
One can see that the function $\eta(\xb,t)=\phi(\xb +Ut e_1, t)\equiv \phi(x +Ut, y,z, t)$
possesses the same properties as in
\eqref{w-eq-U-def} and solves the problem
\begin{equation*}
\begin{cases}
\partial_t^2\eta=\Delta \eta & \text { in } \realsthree_+,\\
\eta(0)=\phi_0;~~\eta_t(0)=\phi_1+U\partial_x\phi_0, \\
\Dn \phi = h^*(\xb,t) & \text{ on } \realstwo_{\{(x,y)\}},
\end{cases}
\end{equation*}
where  $ h^*(\xb,t)=h(\xb +Ut e_1, t)$ which is also belong to
$L_2(0,T;  L_2(\R^2))$, $\forall\, T>0$.
\par
Our goal is to estimate   the time derivative $\eta_t$ on $z=0$.
Clearly, the a priori regularity of $\eta_t $  does not allow to infer any notion of trace.
 However, we will be able to show that
\begin{equation}\label{trace-reg-est}
\int_0^T||\gamma[\eta_t](t)||^2_{H^{-1/2} (\R^2)}dt\le C_T\left( || \eta_0||_{H^1(\realsthree_+)}^2+||\eta_1||^2+
 \int_0^T || h^*(t) ||^2_{L_2 (\R^2)} dt \right).
\end{equation}
Since $\eta_t(\xb,t)\equiv \psi(\xb +Ut e_1, t)$,
we can make inverse change of variable and obtain the statement  of Lemma~\ref{le:FTR}.
\par
In order to prove (\ref{trace-reg-est}) we  apply a principle of superposition with respect to the initial and boundary data.
\par
\noindent{\bf Step 1}: Consider $h^* =0 $.
Here we apply Theorem 3 in \cite{miyatake1973}  with $k =0$.
This yields
\begin{multline}\label{ic}
||\eta(t)||^2_{H^1(\R^3_+)} + ||\eta_t(t)||^2
+ ||\gamma [\eta_t]||^2_{L_2(0,T; H^{-1/2}(\R^2)) } + ||\gamma [\eta] ||^2_{L_2(0,T; H^{1/2}(\R^2))} \\
\leq C_T\left[ ||\eta(0)||^2_{H^1( \R^3_+)} + ||\eta_t(0)||^2\right].
\end{multline}
 \noindent{\bf Step 2: } Consider  zero initial data and  arbitrary $ h^* \in L_2((0, T) \times \R^2 ) $.
We claim that the following estimate holds:
 \begin{equation}\label{trace}
 ||\gamma [\eta]_t||^2_{L_2(0,T; H^{-1/2}(\R^2) } + ||\gamma [\eta] ||^2_{L_2(0,T; H^{1/2}(\R^2))}
\leq C_T\big[ || h^*||^2_{L_2(\Sigma)}  \big].
\end{equation}
The proof of the estimate in (\ref{trace}) depends on the fact that the problem is defined on a half-space.

Since  $\phi_0, \phi_1 =0 $ we then take the Fourier-Laplace transform:
\begin{align*} t \to &~ \tau=\xi + i\sigma,~~~
(x,y) \to  i \mu = i(\mu_1,\mu_2),
\end{align*}
with $\xi$ fixed and sufficiently large.

Now, let $\widehat{\eta} = \widehat {\eta}(z, \mu, \sigma)$ be the Fourier-Laplace transform of $\eta$ in $x$, $y$ and $t$, i.e.
$$
\widehat {\eta}(z, \mu, \tau) = \frac{1}{(2\pi )^2}\int_{{\R}^2} dx dy
\int_{0} ^{+\infty} dt e^{-\tau t} \cdot e^{-i (x\mu_1+y\mu_2)}\cdot
\eta (x,y,z,t), \quad Re\tau >0.
$$
This yields the equation \begin{equation*}
\widehat{\eta}_{zz}=(|\mu|^2+\xi^2-\sigma^2+2i\xi\sigma)\widehat{\eta},
\end{equation*}
 with transformed boundary condition,
\begin{equation*}
\widehat{\eta}_z(z=0)=-\widehat{h^*}(\mu,\tau),~~ \tau=\xi + i\sigma.
\end{equation*}
Solving the ODE in $z$ and choosing the solution which decays as $z \to +\infty$,
 we have
 \begin{equation*} \widehat{\eta}(z,\mu,\tau) = \frac{1}{\sqrt{s}}\widehat{h^*}(\mu,\tau)\exp(-z\sqrt{s}),
 \end{equation*}
 with
 \[
 s \equiv |\mu|^2+\tau^2=\big( |\mu|^2-\sigma^2+\xi^2\big)+2i\xi \sigma;
  \]
  where the square root $\sqrt{s}$ is chosen such that
$ Re\sqrt{\tau^2 +|\mu|^2}>0$ when $Re~ \tau >0$.
  On the boundary $z=0$ we have $\widehat{\eta}(z=0,\mu,\tau)=\frac{1}{\sqrt{s}}\widehat{h^*}(\mu,\tau)$. Taking the time derivative amounts to premultiplying by $\tau=\xi + i \sigma$, and with a slight abuse of notation, we have \begin{equation*}
\widehat{\eta}_t(z=0,\mu,\tau)=\frac{\tau}{\sqrt{s}} \widehat{h^*}(\mu,\tau),
~~ \tau=\xi + i \sigma.
\end{equation*}
Denoting the multiplier $\dfrac{\xi+i\sigma}{\sqrt{s}} \equiv m(\sigma,\mu)$, we can infer the trace regularity of $\eta$ from $m(\sigma,\mu)$:
\begin{align*} |m(\xi,\sigma,\mu)| = &~\Big| \dfrac{\xi+i\sigma}{\sqrt{s}}\Big|
= \dfrac{\sqrt{\xi^2+\sigma^2}}{(|s|^2)^{1/4}} \\
=&~\dfrac{\sqrt{\sigma^2+\xi^2}}{\big((|\mu|^2-\sigma^2)^2+\xi^4+2\xi^2\sigma^2+2|\mu|^2\xi^2\big)^{1/4}}\; .
\end{align*}

\begin{lemma}\label{le:m}
  We have the following estimate
\begin{equation}\label{m-est}
     |m(\xi,\sigma,\mu)|\le 2 \left[ 1+\frac{|\mu|^2}{\xi^2}\right]^{1/4}~~~
     \forall~ (\xi;\sigma,\mu)\in \R^4,~~ \xi\neq 0.
\end{equation}
\end{lemma}
\begin{proof}
First, we can take each of the arguments $\xi$, $\sigma$ and $|\mu|$ to be positive.
 Next we consider the partition of  the first quadrant of the $(|\mu|,\sigma)$-plane as follows: $$
 \begin{cases}(a)~ |\mu| \le  \sigma/\sqrt{2}, \\ (b)~ |\mu| \ge \sqrt{2}\sigma,\\ (c)~\sigma/\sqrt{2}< |\mu| < \sqrt{2}\sigma.\end{cases}.
 $$
In cases (a) and (b) above, we
 can write
\begin{equation*} |m(\xi,\sigma,\mu)| \le \dfrac{\sqrt{\sigma^2+\xi^2}}{\Big(\big| |\mu|^2-\sigma^2\big|^2+\xi^4\Big)^{1/4}}  \le \dfrac{\sqrt{\sigma^2+\xi^2}}{\Big(\sigma^4/4+\xi^4\Big)^{1/4}} \le 2.
\end{equation*}
 In case (c) we have
 \begin{align*}
 |m(\xi,\sigma,\mu)| \le &\dfrac{\sqrt{\sigma^2+\xi^2}}{\big(\xi^4+2\xi^2\sigma^2+2|\mu|^2\xi^2\big)^{1/4}}
 \le \dfrac1{|\xi|^{1/2}}
 \dfrac{\sqrt{\sigma^2+\xi^2}}{\big(\xi^2+2\sigma^2+2|\mu|^2\big)^{1/4}}
  \\
 \le &\dfrac1{|\xi|^{1/2}}
 \dfrac{\sqrt{\sigma^2+\xi^2}}{\big(\xi^2+3\sigma^2\big)^{1/4}}\le
 \dfrac1{|\xi|^{1/2}}
 \big(\sigma^2+\xi^2\big)^{1/4}\le \dfrac1{|\xi|^{1/2}}
 \big(2|\mu|^2+\xi^2\big)^{1/4}.
  \end{align*}
  This implies the  estimate in \eqref{m-est}.
\end{proof}
By the inverse Fourier-Laplace transform we have that
\[
e^{-\xi t}\eta_t (x,y,z=0,t) = \frac{1}{2\pi }\int_{{\R}^2} d\mu
\int_{-\infty} ^{\infty} d\sigma
e^{i\sigma t} \cdot e^{i( x\mu_1+y\mu_2)}\cdot
m(\xi,\sigma,\mu)
\widehat {h^*} (\mu, \xi +i\sigma).
\]
Thus by the Parseval  equality
\[
n(\xi;\eta_t)\equiv\frac{1}{2\pi}\int_0^{+\infty}\|e^{-\xi t}\eta_t (z=0,t)\|^2_{H^{-1/2}(\R^2)}dt =\int_{{\R}^2} d\mu
\int_{-\infty} ^{\infty} d\sigma
\dfrac{|m(\xi,\sigma,\mu)|^2}{ (1+|\mu|^2)^{1/2}}
|\widehat {h^*} (\mu, \xi +i\sigma)|^2
\]
and hence
\[
n(\xi;\eta_t)\le (1+\xi^{-2})^{1/2}\int_{{\R}^2} d\mu
\int_{-\infty} ^{\infty} d\sigma
|\widehat {h^*} (\mu, \xi +i\sigma)|^2
\]
Since
$$
\widehat {h^*} (\mu, \xi +i\sigma) = \frac{1}{ (2\pi )^2 }\int_{{\R}^2} dx dy
\int_{0} ^{+\infty} dt ~e^{-i\sigma t} \cdot e^{-i (x\mu_1+y\mu_2)}\cdot
e^{-\xi t} h^* (x,y,t),~~ \xi>0,
$$
we obtain that
\[
n(\xi;\eta_t)\le 2\pi(1+\xi^{-2})^{1/2} \int_{{\R}^2} d xdy
\int_{0} ^{\infty} d  t~ e^{-\xi t}  |h^* (x,y,t)|^2.
\]
This implies the estimate for
$||\gamma [\eta]_t||^2_{L_2(0,T; H^{-1/2}(\R^2) }$
in \eqref{trace}. To obtain the corresponding bound for
$||\gamma [\eta] ||^2_{L_2(0,T; H^{1/2}(\R^2))}$
we use a similar argument.
\par
The relations in \eqref{ic} and \eqref{trace}
yield
the conclusion of Lemma~\ref{le:FTR}.
\end{proof}

\begin{remark}
In Lemma~\ref{le:FTR} we could also take an arbitrary  smooth domain $\cO$
instead of  $\R^3_{+}$.
Indeed let
$Q = \cO \times (0,T)$ and $\Sigma=\pd\cO  \times (0, T)$.
Assuming  {\it a priori $H^1(Q)$ }  regularity of the solution, then one can show
that $\gamma [\eta_t] \in L_2(0,T; H^{-1/2}(\pd\cO) )$ for all smooth domains.
The a priori $H^1(Q) $ regularity is automatically satisfied
when the Neumann  datum $h^*$ is zero and the initial data are of finite energy ($H^1\times L_2$).
In the case  when $h^*$ is an arbitrary element of $L_2(\Sigma) $,  the corresponding estimate takes the form
\begin{equation}\label{trace1-n}
 ||\gamma [\eta_t]||^2_{L_2(0,T; H^{-1/2}(\R^2) )} + ||\gamma [\eta] ||^2_{L_2(0,T; H^{1/2}(\R^2))}
\leq C_T\big[ || h^*||^2_{L_2(\Sigma)} + ||\eta||^2_{H^1(Q) }   \big].
\end{equation}
The proof of this estimate can be  obtained via microlocal analysis by adopting the argument given in
\cite{l-t-sharp,miyatake1973}.
In our case when $\cO=\R^3_{+}$  we have the estimate in \eqref{trace} which
 does not  contain the term $||\eta||^2_{H^1(Q)}$ and hence it
 can be extended to less regular solutions.
In the case of  general domains
$L_2(\Sigma)$ Neumann boundary data produce in wave dynamics  only
$H^{2/3} (Q) $ solutions   with
less regular (than (\ref{trace1-n})) boundary traces
$\gamma[\eta] \in H^{1/3}(\Sigma) $ and $\gamma[\eta_t] \in H^{-2/3}(\Sigma)$
  (see \cite{tataru} and also \cite{l-t-sharp}). The above result is optimal
  and can not be improved, unless special geometries
  for $\Omega$ are considered. We also note that the above mentioned result  improves upon   \cite{miyatake1973} where the  interior  regularity with $L_2(\Sigma)$ Neumann data  is only  $\eta \in H^{1/2}(Q)$,
   rather than $H^{2/3}(Q)$.   Thus, for general domains  we observe
 an additional  loss, with respect
 to (\ref{trace}), of smoothness for the boundary traces $\gamma [\eta] $ ($1/6 =2/3 - 1/2$ of the derivative).
  \end{remark}
\subsection{Linear Generation of Unperturbed Dynamics}
Our main result in this section is the following assertion:
\begin{proposition}\label{pr:op-bA}
 The operator  $\bA$ given by \eqref{op-A} and \eqref{dom-bA}
 is  maximal, dissipative  and skew-adjoint (i.e. $\bA^*=-\bA$).
 Thus by the Lumer-Phillips theorem (see \cite{Pazy}) $\bA$ generates a strongly continuous \textnormal{isometry group} $e^{\bA t}$ in $Y$.
\end{proposition}
Our calculations for the proof requires some approximation of the domain
$\cD(\bA)$ as a preliminary step.
\subsubsection{Domain Approximation}
We would like to build a family of approximants which  allows us to justify the formal calculus
occuring  in the subsequent  dissipativity and maximality considerations.
Below we concentrate on the more difficult supersonic case $U>1$.
\par
Let us take arbitrary $(\phi, \psi; u, v)$ in $\cD(\bA)$. Then
$$-U \phi_x + \psi  \in H^1(\realsthree_+) , ~~ -U \Dx \psi   - A (\phi+ Nv )     \in L_2(\realsthree_+). $$
Since we also have that $\phi\in H^1(\realsthree_+)$,
 there exist
$h \in H^1 (\realsthree_+)$ and $g \in L_2(\realsthree_+)$  (depending on the element in the domain) such that
\begin{equation}\label{h-g}
    -U \phi_x + \psi -  r \phi   =h \in H^1(\realsthree_+),~~~  -U \Dx \psi   - A(\phi+ Nv )   = g \in L_2(\realsthree_+)
\end{equation}
for every  $r\in\R$.
From the first relation:
$$ \Dx \psi = U \Dx^2 \phi + r \Dx \phi + h_x . $$
Substituting into the second yields
$$ -  U^2 \Dx^2  \phi - U r \Dx \phi - U h_x  - A(\phi +Nv )    = g. $$
This is equivalent to
\begin{equation}\label{wave3}
(U^2 -1) \Dx^2 \phi +U r \phi_x  - (\Delta_{y,z}+\mu)\phi =-g - U h_x = f  \in
L_2(\realsthree_+)
 \end{equation}
 with the boundary conditions $\Dn \phi = v_{ext}$, where $v\in H^2_0(\Om)$.
 By the trace theorem there exists $\eta \in H^{7/2}(\realsthree_+)$ such that
  $\Dn \eta = v_{ext}$. Therefore we can represent
  \[
  \phi=\phi_*+\eta,
  \]
 where $\phi_*\in  H^1(\realsthree_+)$ is a (variational) solution to
 the problem
   \begin{align}\label{wave3*}
 &   (U^2 -1) \Dx^2 \phi_* +U r \phi_{*x}  - (\Delta_{y,z}+\mu)\phi_* = f_*  \in
L_2(\realsthree_+), \\
&\Dn \phi_* = 0,\notag
   \end{align}
 and
 \[
 f_*=f -(U^2 -1) \Dx^2 \eta -U r \eta_x  + (\Delta_{y,z}+\mu)\eta.
 \]
Due to zero Neumann conditions the problem in \eqref{wave3*} can be extended
to the same equation in   $\R^3$.
Therefore
we can apply  Fourier transform in all variables. This gives us
$(x\leftrightarrow \om,~~ (y,z)\leftrightarrow k\in\R^2):$
\[
(-c_U\om^2+  i r U \om +|k|^2+\mu)\widehat\phi_*= \widehat f_*\in L_2,
\]
where $c_U=(U^2 -1)$.
Since\footnote{In the subsonic case ($c_U<0$) the equation in \eqref{wave3*}
is elliptic and  we can obtain better estimate. }
\[
|-c_U\om^2+ i U r \om +|k|^2+\mu|^2 =r^2U^2 \om^2 +(|k|^2+\mu-c_U\om^2)^2\ge c_0(|k|^2+\om^2+1)
\]
with some $c_0>0$ for appropriate $r=r(c_U,\mu)>0$,
the formula above leads to the solutions $\phi_*$ in $H^1(\realsthree_+)$ for
{\em every} $f_*\in L_2(\realsthree_+)$.
Moreover if $f_*\in H^s(\realsthree_+)$
is such that its even (in $z$) extension on $\R^3$ has the same smoothness,
then $\phi_*\in H^{1+s}(\realsthree_+)$ for every $s\ge 0$; hence for  $s<3/2$
we can define a continuous affine map
  $$
 \tau:H^s(\realsthree_+) \to H^{s+1}(\realsthree_+), ~~\mbox{where}~~  \tau(f)=\phi,
 $$
 where $\phi$ solves \eqref{wave3} with the Neumann boundary condition
 $\Dn \phi = v_{ext}\in H^2(\R^2)$.
\par
Thus, in order to find an approximate domain $\mathcal{D}_n$ it suffices to solve
(\ref{wave3}) with the right hand side in $H^1(\R^3_+)$.
Hence, we are looking for $\phi^n \in H^2(\R^3_+)$ such that
\begin{equation}\label{waven}
(U^2 -1) \Dx^2 \phi^n - \Delta_{y,z} \phi^n +\mu \phi^n +U r \phi_x^n = -g^n - U h^n_x  \in H^1(\R^3_+),
 \end{equation}
 and $\Dn \phi^n = v_{ext}$ for all $n$, and
where $h^n \rightarrow h $ in  $H^1(\R^3_+) $ and $g^n \rightarrow g $ in
$L_2(\R^3_+)$. Here $h$ and $g$ are given by  \eqref{h-g}.
\par
Solving  equation (\ref{waven})  with right hand side $ -g^n - Uh^n_x  \in H^1(\R^3_+)  $  gives solution
\[
\phi^n \in H^2(\R^3_+),~~ \phi^n_x \in H^1(\R^3_+).
\]
We then define $\psi^n \equiv U \phi^n_x +r\phi^n + h^n\in H^1(\R^3_+)$.
Then
$$
U \psi^n_x + A( \phi^n+Nv) = U^2 \phi^n_{xx}+ U h^n_x + A (\phi^n+Nv)+rU\phi_x^n =
 -g^n \rightarrow g = U \psi_x + (-\Delta +\mu) \phi
 $$
as desired. Additionally, since $\tau$ is affine and bounded, and each $\phi^n$
has the same boundary condition, we can conclude  that $$||\phi^n-\phi||_{H^1(\realsthree_+)}
=
 ||\tau(g^n + Uh^n_x) -\tau(g + Uh_x) ||_{H^1(\realsthree_+)} \le ||\phi^n-\phi+U(h^n_x-h_x)||_{L_2(\realsthree_+)} \to 0.$$ We may proceed similarly on $\psi^n = U\phi_x^n+r\phi^n+h^n$.
Thus, we have obtained
\begin{lemma}[{\bf Domain Approximation}]\label{l:ap}
For any  $y = (\phi,\psi;u,v) \in \cD(\bA) $ there exist  approximants $\phi^n \in H^2(\R^3_+) , \psi^n \in H^1(\R^3_+) $ such that
$y^n = (\phi^n, \psi^n; u, v) \in \cD(\bA) $ and $y^n \rightarrow y $ in $Y$. Moreover
$$ U \psi^n_x + A( \phi^n+Nv)   \rightarrow   U \psi_x + A (\phi +Nv), ~in~ L_2(\R^3_+ ), $$
$$ \psi^n -U \phi^n_x \rightarrow \psi - U \phi_x , ~in~ H^1(\R^3_+).
$$
As a consequence
$$
(\bA y^n, y^n )_Y \rightarrow (\bA y,y )_Y ~~\mbox{for all~~ $y \in \cD(\bA)$.}
$$
\end{lemma}

\subsubsection{Dissipativity}
The above approximation Lemma  allows us to perform calculations on smooth functions.
\par
Let $y^n\in \cD(\bA)$ be the sequence of approximants as in  Lemma \ref{l:ap}.
First, we perform the dissipativity calculation on these approximants (which allows us to move $A^{1/2}$ freely on flow terms $\phi^n$ and $\psi^n$):
\begin{align*}
(\bA y^n,y^n)_Y
= &\begin{pmatrix}\left[\begin{array}{c} -U\partial_x \phi^n+\psi^n  \\-U\partial_x\psi^n-A(\phi^n+Nv)\\ v \\ -\cA u+N^*A\psi^n
\end{array}\right], \left[\begin{array}{c} \phi^n \\\psi^n\\ u\\ v
\end{array}\right]
\end{pmatrix}_{ \cD(A^{1/2})\times L_2(\Omega)\times \cD(\cA)^{1/2}\times L_2(\Omega)}\\[.25cm]
\nonumber
=&~(A^{1/2}(\psi^n-U\partial_x\phi^n),A^{1/2}\phi^n)-(U\partial_x\psi^n+A(\phi^n+Nv),\psi^n)\\[.25cm]
\nonumber
&~+<\cA^{1/2}v,\cA^{1/2}u>-<\cA u-N^*A\psi^n,v>\\[.25cm]
\nonumber
=&~-U(A^{1/2}\partial_x\phi^n,A^{1/2}\phi^n)-(U\partial_x\psi^n+ANv,\psi^n)+<N^*A\psi^n,v>.
\end{align*}
One can see that
\[
(A^{1/2}\partial_x\phi^n,A^{1/2}\phi^n)=\int_{\R^3_+}\nabla\partial_x\phi^n \cdot \nabla\phi^n=
  \frac12 \int_{\R^3_+}\partial_x|\nabla\phi^n|^2=0
\]
Similarly $(U\partial_x\psi^n,\psi^n)=0$.
Therefore
\[
(\bA y^n,y^n)_Y
= -<ANv,\psi^n>+<N^*A\psi^n,v>
= 0,
\]
 Furthermore, by the convergence result in Lemma \ref{l:ap}, we have that for all $y \in \cD(\bA)$ $$(\bA y, y)_Y=\lim_{n \to \infty}(\bA y^n,y^n)_Y= 0.$$
This gives that {\em both} operators $\bA$ and $-\bA$ are dissipative.
\subsubsection{Maximality}
 In this section we prove the maximality of the operators $\bA$
 and $-\bA$. For this it is sufficient   to show
 $ \mathscr{R}(\lambda-\bA) = Y$ for every $\la\in\R$, $\la\neq 0$,
 i.e. for a given $F = (\phi', \psi';  u', v')$, find a $ V \in \mathscr{D}(\bA)$ such that $(\lambda-\bA)V = F.$ Writing this as a system, we have \begin{align}\label{full-eq-lamb}
 \begin{cases}
 \lambda \phi + U \partial_x \phi - \psi &= \phi' \in \cD(A^{1/2}),  \\
 \lambda \psi + U \partial_x \psi + A(\phi + Nv)  & = \psi' \in L_2(\realsthree_+), \\
 \lambda u - v  &= u' \in \cD(\cA^{1/2}),\\
 \lambda v + \mathscr{A}u - N^*A\psi & = v' \in L_2(\Omega),
 \end{cases}
 \end{align}
 (recalling that $Nv$ is taken to mean $Nv_{ext}$ where $v_{ext}$ is the extension by zero outside of $\Omega$).
 \par

 In the space $Y$ we rewrite \eqref{full-eq-lamb} in the form
 \begin{equation}\label{a-form}
    a(V,\widetilde{V})=(F,\widetilde{V})_Y,
 \end{equation}
 where for
$ V = (\phi, \psi; u, v)$ and  $\widetilde{V} =(\tilde\phi, \tilde\psi;  \tilde u, \tilde v)$
 we denote
 \begin{align*}
   a(V,\widetilde{V}) =&
 (\lambda \phi + U \partial_x \phi - \psi, A\tilde\phi)_{\R^3_+}  \\
& +( \lambda \psi + U \partial_x \psi + A(\phi + Nv),\tilde\psi)_{\R^3_+}   \\
 & +(\lambda u - v, \cA \tilde u)_\Omega+
 ( \lambda v + \mathscr{A}u - N^*A\psi,\tilde v)_\Omega.
 \end{align*}
 Let $\{\eta_k\} \times \{e_k\}$ be a sufficiently smooth basis in $\cD(A^{1/2}) \times \cD(\cA^{1/2})$.
 We define an $N$-approximate solution to \eqref{a-form} as an element
 \[
 V_N\in Y_N\equiv\mbox{Span}\;\left\{ (\eta_k,\eta_l;e_m,e_n)\, :\; 1\le k,l,m,n\le N
 \right\}
 \]
 satisfying the relation
 \begin{equation}\label{a-form-N}
    a(V_N,\widetilde{V})=(F,\widetilde{V})_Y,~~~\forall~ \widetilde{V}\in Y_N.
 \end{equation}
 This can be written as a linear $4N\times 4N$ algebraic equation.
  Calculations on (smooth) elements $V$ from $Y_N$ gives
 \begin{align*}
   a(V, V) = \lambda \big\{ ||A^{1/2}\phi||^2 + ||\psi||^2+
   ||\cA^{1/2}u||+||v||^2\big\}
 \end{align*}
 This implies that for every $\lambda\neq0$ the matrix which corresponds
 to \eqref{a-form-N} is non-degenerate, and therefore there exists
 a unique approximate solution $V_N=(\phi^N, \psi^N; u^N, v^N)$. Moreover we have that
 \[
   a(V_N,V_N)=(F,V_N)_Y
 \]
 which implies the a priori estimate
 \[
  ||A^{1/2}\phi^N||_{\R^3_+} ^2 + ||\psi^N||_{\R^3_+} ^2+
   ||\cA^{1/2}u^N||_\Om^2+||v^N||^2_\Om\le \frac{1}{\lambda^2}\|F\|_Y^2
 \]
Thus $\{V_N\}$ is weakly compact in $Y$. This allows us  to
 make limit transition in
 \eqref{a-form-N} to obtain the equality
 \[
   a(V,\widetilde{V})=(F,\widetilde{V})_Y,~~~\forall\, \widetilde{V}\in Y_M,~~
   \forall\, M,
 \]
 for some $V\in Y$. Thus \eqref{full-eq-lamb} is satisfied
 in the sense distributions. This proves maximality of both operators
 $\bA$ and $-\bA$.
\par
Since both operators
 $\bA$ and $-\bA$ are maximal and dissipative, we can apply
\cite[Corollary 2.4.11. p.25]{cas-har98} and conclude that
the operator $\bA$ is skew-adjoint with respect to Y.
This completes the proof of Proposition~\ref{pr:op-bA}.
\par
 The fact that $\bA$ is skew-adjoint  simplifies  calculations later in the treatment.
 In what follows, we use $\cD(\bA)$ and $\cD(\bA^*)$ interchangeably.
 \begin{remark}
 To conclude this section, we mention that for $y_0 \in Y$ the $C_0$-group $e^{\bA t}$ generates  a mild solution $y(t)=e^{\bA t}y_0$ to the PDE problem given
 in \eqref{sys1*} without the boxed term (see also \eqref{sys2*} below with $k=0$).
We also note that adding a linearly bounded perturbation to the dynamics will not affect the generation of a $C_0$ group. In particular, the addition of internal damping for the plate of the form $k u_t,~k>0$ on the LHS of (\ref{plate}) does not affect generation and we can construct $C_0$-group $T_k(t)$ which corresponds to the problem
 \begin{equation}\label{sys2*}\begin{cases}
(\partial_t+U\partial_x)\phi =\psi& \text{ in } \realsthree_+ \times(0,T),
\\
(\partial_t+U\partial_x)\psi=\Delta \phi-\mu \phi&\text{ in } \realsthree_+\times (0,T),
\\
\Dn \phi = -u_t \cdot \mathbf{1}_{\Omega}(\xb) & \text{ on }
\realstwo_{\{(x,y)\}} \times (0,T),
\\
u_{tt}+k u_t+\Delta^2u=\gamma[\psi]&\text{ in }  \Om \times (0,T),\\
u=\Dn u = 0 & \text{ in } \pd\Om \times (0,T).\\
\end{cases}
\end{equation}
Moreover, this damping term $ku_t$ does not alter $\cD(\bA)$ or $\cD(\bA^*)$, and hence the analysis to follow
 concerning the full dynamics ($\bA+\bP$, and the addition of nonlinearity) is valid presence of interior damping.
We plan use this observation in future studies of this model.
\end{remark}
\subsection{Variation of Parameters and Perturbed Linear Dynamics}
We begin with some preparations.

\subsubsection{Preliminaries}

We would like to introduce a perturbation to the operator $\bA$ which will produce the non-monotone flow-structure problem above. For this we define an operator $\bP:Y \to \mathscr{R}(\bP)$ as follows:
\begin{equation}\label{op-P}
\bP\begin{pmatrix}\phi\\\psi\\u\\v \end{pmatrix}= \bP_{\#}[u]\equiv\begin{pmatrix}0\\-UAN\partial_x u\\0\\0\end{pmatrix}
\end{equation}
Specifically, the problem in (\ref{sys1*}) has the abstract Cauchy formulation:
\begin{equation*}
y_t = (\bA +\bP)y, ~y(0)=y_0,
\end{equation*}
where $y_0 \in Y$ will produce semigroup (mild) solutions to the corresponding integral equation, and $y_0 \in \cD(\bA)$ will produce classical solutions. To find solutions to this problem, we will consider a fixed point argument, which necessitates interpreting and solving the following inhomogeneous problem,  and then producing the corresponding estimate on the solution:
\begin{equation} \label{inhomcauchy}
y_t = \bA y +\bP_\# \overline{u}, ~t>0, ~y(0)=y_0,
\end{equation}
for a given $\overline{u}$. To do so, we must understand how $\bP$ acts on $Y$
(and thus $\bP_\#$ on $H_0^2(\Om)$).
\par
To motivate the following discussion, consider for $y \in Y$ and $z= (\overline\phi, \overline\psi;  \overline u, \overline v)$ the formal calculus (with $Y$ as the pivot space)
\begin{align}\label{p-y-z}
(\bP y, z)_Y =& (\bP_{\#}[u],z)_Y
= -U(AN\partial_x u,\overline \psi)
=-U<\partial_x u, \gamma[\overline \psi]>.
\end{align}
Hence, interpreting the operator $\bP$ (via duality) is contingent upon the ability to make sense of $\gamma[\overline \psi]$, which \textit{can} be done if $\gamma[\overline \psi] \in H^{-1/2}(\Omega)$. In what follows, we show that the trace estimate on $\psi$ for mild solutions of (\ref{inhomcauchy}) allows us to justify the formal energy method (multiplication of (\ref{inhomcauchy}) by the solution $y$) in order to perform a fixed point argument.
\par
To truly get to the heart of this matter, we must interpret the following variation of parameters statement for $\overline{u} \in C(\R_+; H^2_0(\Om))$ (which will ultimately be the solution to (\ref{inhomcauchy})):
\begin{equation}\label{solution1}
y(t)=e^{\bA t}y_0 + \int_0^t e^{\bA(t-s)}\bP_{\#} [\overline{u}(s)] ds.
\end{equation}
To do so, we make use of the work in \cite{redbook} and write (with some $\la\in\R$, $\la\neq 0$):
\begin{equation}\label{solution2}
y(t)=e^{\bA t}y_0 + (\lambda  -\bA)\int_0^t e^{\bA(t-s)}(\lambda  - \bA)^{-1}\bP_{\#} [\overline{u}(s)]  ds,
\end{equation}
 initially interpreting this solution as an element of $[\cD(\bA^*)]'=[\cD(\bA)]'$, i.e., by considering the solution $y(t)$ in (\ref{solution2}) above acting on an element of $\cD(\bA^*)$.
 \subsubsection{Abstract Semigroup Convolution}
 At this point we cast the discussion of the perturbation $\bP$ (acting outside of $Y$) into the context of abstract boundary control. By doing this we simplify and distill our exposition, and moreover, provide a context for further boundary control considerations. For this discussion, we select and cite some results from \cite[pp.645-653]{redbook} which will be used in this section.
 \par
 Let  $X$ and $\cU$ be  reflexive Banach spaces. We assume that
 \begin{enumerate}
\item[(C1)] $A$ is a linear operator which generates a strongly continuous semigroup $e^{At}$ on $X$.
\item[(C2)] $B$ is a linear continuous operator from $\cU$ to $[\cD(A^*)]'$ (duality with respect to the pivot space $X$), or equivalently, $(\lambda - A)^{-1}B \in \mathscr{L}(\cU,X)$ for all $\lambda \in \rho(A)$.
 \end{enumerate}
  For fixed $0<T<\infty$ and $u\in L_1(0,T;\cU)$ we  define  the convolution operator
  $$
  (Lu)(t) \equiv \int_0^t e^{A(t-s)}Bu(s) ~ds,~~0 \le t \le T,
  $$
  corresponding to the mild solution
  $$
  x(t) =e^{At}x_0+(Lu)(t),~~0 \le t \le T,
  $$
  of the abstract inhomogeneous equation
  $$
  x_t = Ax + Bu \in [\cD(A^*)]',~x(0)=x_0,
  $$
  with the input function $Bu(t)$.

\begin{theorem}[{\bf Inhomogeneous Abstract Equations}]\label{dualityequiv}
Let $X$ and $\cU$ be reflexive Banach spaces and
the  conditions in $(C1)$ and $(C2)$ be in force.
Then
\begin{enumerate}
\item The semigroup $e^{At}$ can be extended to the space $[\cD(A^*)]'$.
 \item
 $L$ is continuous from $L_p(0,T;X)$ to $C(0,T; [\cD(A^*)]')$ for every $p \in [1,\infty]$.
\item If $u \in C^1(0,T;X)$, then $Lu \in C(0,T;X)$.
 \item The condition
\begin{enumerate}
\item[(C3)] There exists a constant $C_T>0$ such that
\begin{equation*}
\int_0^T ||B^*e^{A^*t}x^*||^2_{\cU^*} dt \le C_T||x^*||^2_{X^*},
~~\forall\, x^*\in \cD(A^*)\subset X^*,
\end{equation*}
\end{enumerate}
  is equivalent to the regularity property
 $$L:~L_2(0,T;\cU) \to C(0,T;X)~~ \text{is continuous,}
 $$ i.e., there exists a constant $k_T>0$ such that
 \begin{equation}\label{conv-reg}
 ||Lu||_{C(0,T;X)} \le k_T ||u||_{L_2(0,T;\cU)}.
 \end{equation}
 \item Lastly, assume additionally that $A$ generates a strongly continuous group $e^{At}$ (e.g., if $A$ is
 skew-adjoint) and suppose that $L:L_2(0,T;\cU) \to L_2(0,T;X)$ is continuous. Then $(C3)$ is satisfied and thus we have the estimate in \eqref{conv-reg} in this case.
\end{enumerate}
\end{theorem}

\subsubsection{Application of the Abstract Scheme}

We now introduce the auxiliary space which will be needed in the proof of the next lemma: $$
Z \equiv \left\{y =\begin{pmatrix} \phi \\ \psi \\ u \\ v \end{pmatrix}\in Y :
-U\pd_x \phi+\psi \in H^{1}(\R^3_+) \right\}
$$
endowed with the norm
\[
\|y\|_Z=\|y\|_Y+\|-U\pd_x \phi+\psi\|_{H^{1}(\R^3_+)}.
\]
One can see that $Z$ is dense in $Y$.
We also note that by Lemma \ref{domain}, $\cD(\bA) \subset Z$
and thus   $Z' \subset [\cD(\bA)]'$ with continuous embedding.
\begin{lemma}\label{le:P-op}
The operator $\bP_{\#}$ given by \eqref{op-P}
 is a bounded linear mapping  from $H^2_0(\Om)$ into $Z'$.
 Moreover,  the following estimates are in force:
 \begin{equation}\label{y-to-z1}
   ||\bP_{\#} [u]||_{Z'}\le C_U \|u\|_{H^2(\Om)}, ~~~\forall\, u\in H^2_0(\Om),
 \end{equation}
 and also (with  $\la\in\R$, $\la\neq 0$):
  \begin{equation}\label{y-to-z2}
   ||(\la-\bA)^{-1}\bP_{\#}[ u]||_{Y}\le C_{U,\la} \|u\|_{H^2(\Om)}, ~~~\forall\, u\in H^2_0(\Om).
 \end{equation}
 In the latter case we understand
$(\la-\bA)^{-1}\,  : [\cD(\bA)]'\mapsto Y$  as
the inverse to  the operator $\la-\bA$ which is extended to a mapping  from
$Y$ to  $[\cD(\bA)]'$. We also have that \eqref{y-to-z1} and \eqref{y-to-z2}
imply that $\bP$ maps $Y$ into $Z'$ and
 \begin{equation*}
   ||\bP y ||_{Z'}\le C \|y\|_{Y}  ~~\mbox{and}~~
||(\la-\bA)^{-1}\bP y||_{Y}\le C \|y\|_{Y},~~ \forall\, y\in Y.
 \end{equation*}
\end{lemma}
\begin{proof}
 For $u\in H^2_0(\Om)$ and $\overline{y}=(\overline{\phi},\overline{\psi};\overline{u},\overline{v}) \in Z$,
 from \eqref{p-y-z} we
have
\[
|(\bP_{\#}[ u], \overline y)_Y| =~U |(AN\partial_x  u, \overline \psi)|
= U |<\partial_x u, \gamma[\overline\psi]>|.
\]
Since $\ga[\overline\psi] = \ga[-U\pd_x \overline\phi+\overline\psi]+
U\pd_x\ga[\overline\phi]$, we have from the trace
theorem  that
\begin{equation*}
    ||\gamma[\overline\psi]||_{H^{-1/2}(\R^2)}\le C \left[||(-U\pd_x \overline\phi+\overline\psi)||_{H^{1}(\R_+^3)} +\|\overline
\phi||_{H^{1}(\R^3_+)}\right]\le C\|\overline y\|_{Z} .
\end{equation*}
Therefore
\[
|(\bP_{\#} [u], \overline y)_Y|\le
 C(U)  ||\pd_x u||_{H^{1/2}(\Omega)}||\gamma[\overline\psi]||_{H^{-1/2}(\R^2)}
\le  ~ C(U)  || u ||_{H^{2}(\Omega)} \| \overline y\|_{Z}.
\]
Thus
\[
||\bP_{\#} [u]||_{Z'} = \sup\big\{
|(\bP_{\#} [u], \overline y)_Y|\, : \overline y \in Z,~||\overline y||_Z =1\big\}
\le  C(U)  || u ||_{H^{2}(\Omega)},
\]
which implies \eqref{y-to-z1}.
 The relation in \eqref{y-to-z2} follows from  \eqref{y-to-z1} and
the boundedness of the operator
$(\la-\bA)^{-1}\, : [\cD(\bA)]'\mapsto Y$.
\end{proof}
We may now consider mild solutions to the problem given in (\ref{inhomcauchy}).
Applying general results on $C_0$-semigroups (see \cite{Pazy}) we arrive at the following assertion.
\begin{proposition}\label{pr:mild}
Let $\overline{u} \in C^1([0,T];H^{2}_0(\Omega))$ and $y_0\in Y$. Then $y(t)$
given by \eqref{solution1} belongs to  $C([0,T];Y)$ and
is a strong solution to \eqref{inhomcauchy} in $[\cD(\bA)]'$, i.e.
in addition we have that
\[
y\in  C^1((0,T);[\cD(\bA)]')
\]
and  \eqref{inhomcauchy} holds in $[\cD(\bA)]'$ for each $t\in (0,T)$.
\end{proposition}
\begin{proof} The result follows from the integration by parts formula
\begin{align*} (L\overline{u})(t) \equiv&~   \int_0^t e^{A (t-s) }  \bP_{\#} [\overline{u}(s)]ds = (\lambda - A  ) \int_0^t e^{A (t-s) }  (\lambda - A )^{-1} \bP_{\#}[\overline{u}(s)] ds\\\nonumber
 =&~ \lambda \int_0^t e^{A (t-s) }   (\lambda - A  )^{-1}\bP_{\#}[\overline{u}(s)] ds -  \int_0^t e^{A (t-s) }
(\lambda - A )^{-1} \bP_{\#}[\frac{d}{ds}\overline{u}(s)] ds\\\nonumber
& -e^{A t} (\lambda -A )^{-1} \bP_{\#}[\overline{u}(0)]
+(\lambda-A)^{-1}\bP_{\#}[\overline{u}(t)]\\\nonumber
\in& ~C(0,T;X),
 \end{align*}
 and by applying \cite[Corollary 2.5, p.107]{Pazy} to show that
$y$ is a strong solution in $[\cD(\bA)]'$.
\end{proof}
Proposition~\ref{pr:mild} implies that $y(t)$ satisfies
the variational relation
\begin{equation}\label{weak-eq}
    \pd_t(y(t),h)_Y=-(y(t), \bA h)_Y+ (\bP_{\#}[\overline{u}(t)],h)_Y,~~~\forall\, h\in \cD(\bA).
\end{equation}
In our context in application of Theorem~\ref{dualityequiv} we have that $X=Y$,
where $Y$ is Hilbert space given by
\eqref{space-Y}, $\cU=H^2_0(\Om)$ and $B=\bP_\#$ is defined by \eqref{op-P}
as an operator from  $\cU=H^2_0(\Om)$ into $Z'$ (see Lemma~\ref{le:P-op}).
One can see from \eqref{p-y-z} that the adjoint operator $\bP_\#^*\, : Z\mapsto H^{-2}(\Om)$ is given
\[
\bP^*_{\#}z=U\big[\partial_x N^*A \overline \psi\big]\big|_\Om = U\partial_x \gamma[\overline \psi]\big|_\Om,
 ~~\mbox{for}~~ \ds z = \begin{pmatrix} \overline \phi \\ \overline \psi \\ \overline u \\ \overline v\end{pmatrix}\in Z.
 \]
Condition $(C3)$ in Theorem~\ref{dualityequiv}(4) is then paraphrased by writing
\begin{equation*}
y\mapsto
 \bP^*_\#e^{\bA^*(T-t)}y\equiv \bP^*_\#e^{\bA t}y_T : ~\text{continuous}~ Y \to L_2\big(0,T;H^{-2}(\Omega)\big),
\end{equation*}
where $y_T=e^{\bA^*T}y=e^{-\bA T}y$ (we use here the fact that $\bA$ is a skew-adjoint generator).
\par
Let us denote by
$e^{\bA t}y_T\equiv  w_T(t) =(\phi(t),\psi(t); u(t), v(t))$
the solution of the linear problem in \eqref{eq-trankt-0} with initial data $y_T$.
Then from the trace estimate in (\ref{trace-reg-est-M}) we have that
 \begin{align*}
||\bP^*_\#e^{\bA t}y_T||^2_{L_2(0,T:H^{-2}(\Omega))} = &~ U\int_0^T ||\partial_x\ga[\psi(t)]||^2_{H^{-2}(\Omega)}dt \\ \nonumber
\le
& C \int_0^T ||\gamma[\psi(t)]||^2_{H^{-1/2}(\Omega)}  dt \le C_{T} \left(E_{fl}(0)+\int_0^T ||v(t)||^2_{L_2(\Omega)}\right)dt \\ \nonumber
\le &  C_{T}\| y_T\|_Y=C_{T}\| e^{-\bA T}y\|_Y  =C_{T}\|y\|_Y.
\end{align*}
In the last equality we also use that
 $ e^{\bA t}$ is a $C_0$-group of isometries.

Now we are fully in a position to use Theorem \ref{dualityequiv}
which leads to the following assertion.

\begin{theorem}[{\bf $L$ Regularity}]
Let $T>0$ be fixed, $y_0 \in Y$ and $\overline u \in C([0,T]; H_0^{2}(\Omega))$.
Then  the mild solution
 \begin{equation*}
 \ds y(t)=e^{\bA t}y_0+L[\overline u](t)\equiv
 e^{\bA t}y_0+\int_0^te^{\bA (t-s)}\bP_{\#}[\overline u(s)] ds
 \end{equation*}
 to problem \eqref{inhomcauchy} in $[\cD(\bA)]'$ belongs to the class $C([0,T]; Y)$
 and enjoys the estimate
\begin{align}\label{followest}
\max_{\tau \in [0,t]} ||y(\tau)||_Y \le&  ||y_0||_Y + k_T ||\overline u||_{L_2(0,t;H^{2}_0(\Omega))},~~~\forall\, t\in [0,T].
\end{align}
\end{theorem}
\begin{remark}
The discussion beginning at Theorem \ref{dualityequiv} and ending with the estimate in \eqref{followest} demonstrates how the perturbation $\bP$ acting outside of $Y$ is regularized when incorporated into the operator $L$; namely, the variation of parameters operator $L$ is a priori only continuous from $L_2(0,T;\cU)$ to $C(0,T;[\cD(\bA ^*)]')$. However, we have shown that the additional ``hidden'' regularity of the trace of $\psi$ for solutions to (\ref{sys1}) allows us to bootstrap $L$ to be continuous from $L_2(0,T;\cU)$ to $C(0,T;Y)$ (with corresponding estimate) via the abstract Theorem \ref{dualityequiv}. This result essentially justifies \textit{formal} energy methods on the equation (\ref{inhomcauchy}) in order to set up a fixed point argument (which will follow in the next section).
\par
For completeness, we also include a direct proof of an estimate of the type \eqref{followest}
 in the Appendix, independent of the abstract boundary control framework presented in Theorem~\ref{dualityequiv}.
\end{remark}

\subsection{Construction of a Generator}
Let  $\bX_t = C\big((0,t];Y\big).$
Now, take $\overline{y}=(\overline{\phi},\overline{\psi};\overline{u},\overline{v}) \in \bX_t$ and $y_0 \in Y$, and introduce the map $\cF: \overline{y} \to y$ given by
\begin{equation*}
y(t) = e^{\bA t}y_0+L[\overline u](t),
\end{equation*}
i.e. $y$ solves \begin{equation*}
y_t=\bA y+\bP_\# \overline u,
~~ y(0)=y_0,
\end{equation*}
in the generalized sense, where $\bP_\#$ is defined in \eqref{op-P}.
 It follows from (\ref{followest})
 that for $\overline y_1, ~\overline y_2 \in \bX_t$
 \begin{align*}
 \|\cF \overline y_1 - \cF \overline y_1\|_{\bX_t}\le & ~k_T ||\overline u_1 - \overline u_2||_{L_2(0,t;H^{2}_0(\Omega))} \\
 \le & ~k_T \sqrt{t}\max_{\tau \in [0,t]}|| \overline u_1 - \overline u_2||_{H^2(\Omega)}
 \le k_T \sqrt{t} ||\overline y_1 - \overline y_2 ||_{\bX_t}.
 \end{align*}
Hence there is $0<t_*<T$ and $q<1$ such that
\[
\|\cF \overline y_1 - \cF \overline y_2\|_{\bX_t}\le q \| \overline y_1 -  \overline y_2\|_{\bX_t}
\]
for every $t\in (0,t_*]$.
This implies that  on the interval $[0,t_*]$ the problem
\begin{equation*}
y_t=\bA y+\bP y, ~~t>0,~
~~ y(0)=y_0,
\end{equation*}
has a local in time unique (mild) solution defined now in $Y$. This above local solution
can be extended to a global solution in finitely many steps by linearity.
Thus there exists a unique function
$y=(\phi,\psi;u,v)\in C\big(\R_+;Y\big)$ such that
\begin{equation}\label{eq-fin}
y(t)=e^{\bA t}y_0+\int_0^te^{\bA(t-s)}\bP [y(s)]ds~~\mbox{ in }~ Y
~~\mbox{for all }~t>0.
\end{equation}
It also follows from the analysis above that
\[
\|y(t)\|_Y\le C_T \|y_0\|_Y,~~~ t\in [0,T],~~\forall\, T>0.
\]
Thus the problem \eqref{eq-fin} generates strongly continuous semigroup
$\widehat{T}(t)$ in $Y$.
Additionally, due to \eqref{weak-eq} we have
\begin{equation*}
    (y(t),h)_Y=(y_0,h)_Y+ \int_0^t\left[
    -(y(\tau), \bA h)_Y+ (\bP[y(\tau)],h)_Y
    \right]d\tau
    ,~~~\forall\, h\in \cD(\bA),
    ~~ t>0.
\end{equation*}
Using the same idea as in subsonic case \cite{jadea12} (which relies on Ball's Theorem \cite{ball} and
ideas presented in  \cite{desh}), we can conclude that
the generator $\widehat{\bA}$ of $\widehat{T}(t)$  has the form
\[
\widehat{\bA}z=\bA z+\bP z,~~z\in\cD(\widehat{\bA})=\left\{z\in Y\,:\; \bA z+\bP z\in Y\right\}
\]
(we note that  the sum $\bA z+\bP z$ is well-defined as an element
in $[\cD(\bA)]'$ for every $z\in Y$). Hence, the semigroup $e^{\widehat \bA t}y_0$ is a generalized solution for $y_0 \in Y$ (resp. a classical solution for $y_0 \in \cD(\widehat \bA)$) to (\ref{sys1*}) on $[0,T]$ for all $T>0$.
\begin{equation}\label{dom-bA-n}
\cD(\bA + \bP) \equiv \left\{ y \in Y\; \left| \begin{array}{l}
-U \Dx \phi + \psi \in H^1(\R^3_+),~\\ -U \Dx \psi  -A (\phi + N (v + U \Dx u )) \in L_2(\R^3_+) \\
v \in \cD(\cA^{1/2} )= H_0^2(\Omega),~
-\cA u + N^* A \psi \in L_2(\Omega) \end{array} \right. \right\}
\end{equation}
Now we can conclude the proof of Theorem \ref{th:lin} in the same way as in
\cite{jadea12} by considering bounded perturbation of generator of the
term
\[
C(y) = ( 0, \mu \phi; 0,0).
\]
Indeed, the
function $y(t)$ is a generalized solution corresponding to the generator $\bA + \bP$ with the  domain defined in (\ref{dom-bA-n}). This proves the first statement in Theorem
 \ref{th:lin}. As for the second statement  (regularity), the invariance of the domain $\cD(\bA + \bP ) $ under the flow
 implies that solutions originating in $Y_1$ will remain in  $C([0, T]; Y_1)$. After identifying $ \phi_t = \psi - U \Dx \phi $
 one translates the membership in the domain into membership in $Y_1$.  Elliptic regularity applied to biharmonic operator yields  the precise regularity  results  defining {\it strong } solutions.
 The proof of Theorem \ref{th:lin} is thus completed.

\subsection{Nonlinear Semigroup and Completion of \\the Proof of Theorem \ref{th:nonlin}}
To prove Theorem \ref{th:nonlin} we also use the same idea
as in \cite{jadea12}. As an intermediate step we obtain the following theorem.
\begin{theorem}\label{th:3.9}
Let $\mathscr{F}(y)=(0,F_1(\phi);0, F_2(u))$ where
 \[
  F_1:H^1(\realsthree_+) \to L_2(\realsthree_+)~~\mbox{and}~~ F_2:  H^2_0(\Omega) \to L_2(\Omega)
 ~~\mbox{ are  locally Lipschitz}
 \]
 in the sense that every $R>0$ there exists $c_R>0$ such that
 \[
 \|F_1(\phi)-F_1(\phi^*)\|_{\R_+^3}\le c_R \|\phi-\phi^*\|_{1,\R_+^3}
 ~~\mbox{and}~~  \|F_2(u)-F_1(u^*)\|_{\Om}\le c_R \|u-u^*\|_{2,\Om}
 \]
 for all $\phi,\phi^*\in H^1(\realsthree_+)$  and $u, u^*\in H^2_0(\Omega)$
 such that $\|\phi\|_{1,\R_+^3}, \|\phi^*\|_{1,\R_+^3}, \|u\|_{2,\Om}, \|u^*\|_{2,\Om}\le R$.
Then the equation
\begin{equation}\label{abstractode}
y_t=(\bA+\bP)y+\mathscr{F}(y), ~y(0)=y_0 \in Y
\end{equation}  has a unique local-in-time generalized solution $y(t)$
(which is also weak).
Moreover, for $y_0 \in \cD (\bA + \bP )$, the corresponding solution is strong.
\par
In both cases, when $t_{\text{max}}(y_0)< \infty$, we have that $||y(t)||_Y  \to \infty$ as $t \nearrow t_{\text{max}}(y_0)$.
\end{theorem}
\begin{proof}
This is a direct application of Theorem 1.4~\cite[p.185]{Pazy}
and  localized version of Theorem 1.6~\cite[p.189]{Pazy}.
\end{proof}
 In order to guarantee global solutions,
one must have more information on the nature of the nonlinear term.
The following result  provides  relevant abstract  conditions imposed on nonlinear terms which can be verified:
\begin{theorem}\label{abstract}
We assume that $f$ is locally Lipschitz from
 $H^2_0(\Omega)$ into $L_2(\Om)$ and there exists
$C^1$-functional $\Pi(u)$ on  $H^2_0(\Omega)$ such that
$f$ is a Fr\'echet derivative of $\Pi(u)$,
$f(u)=\Pi'(u)$. Moreover we assume that  $\Pi(u)$
 is  locally bounded on  $H^2_0(\Omega)$, and there exist $\eta<1/2$
 and $C\ge 0$ such that
\begin{equation}\label{8.1.1c1}
 \eta \|\Delta u\|_\Om^2 +\Pi(u)+C \ge 0\;,\quad \forall\, u\in  H^2_0(\Omega)\;.
\end{equation}
Then the generalized solution in  (\ref{abstractode}) is global, i.e. $ t_{max} = \infty $.
\end{theorem}
\begin{proof}
The relation in \eqref{8.1.1c1} implies that
the full energy $\cE(t)$ defined in \eqref{energies} admits the estimates
\[
\cE(t)\ge c_0\left(  \|\psi\|_{\R^2_+}^2 +\|\phi\|_{\R^2_+}^2 +\|u_t\|_\Om^2 + \|\Delta u\|_\Om^2 \right) -c_1
\equiv c_0 \|y(t)\|_Y^2-c_1
\]
for some positive  $c_i$. Therefore
using the energy relation in \eqref{energyrelation} on the existence time interval
and the flow trace regularity (see Lemma~\ref{le:FTR}) we can conclude
that
\begin{align*}
c_0\|y(t)\|_Y^2\le & c_1+ \cE(0)+\int_0^t\|\Delta u(\tau)\|_\Om \|\psi(\tau)\|_{H^{-1/2}(\R^2_+)} d\tau
\\
\le & C(y_0)+C_T\int_0^t\|y(\tau)\|_Y^2 d\tau.
\end{align*}
Thus the solution $y(t)$ cannot blow up at finite time, i.e., $ t_{max} = \infty $.
\end{proof}

Thus, in order to complete the proof of the first part of  Theorem \ref{th:nonlin} one needs to verify  that the nonlinear forcing
$f(u)$  given in Assumption \ref{as} comply with the requirements of  Theorem \ref{abstract}.

\subsubsection{Verification of the Hypotheses}

We note  the examples of forcing terms described above satisfy
conditions of  Theorem \ref{abstract}.

{\it Step 1: Kirchhoff model}.
Indeed,
in the case of the {\it Kirchhoff model}, the  embeddings
$H^s(\Om)\subset L_{\infty}(\Om)$ for $s>3/2$
implies that
\begin{equation}\label{loc-lip-f}
\|f(u_1)-f(u_2)\|_\Om\le c_R \|u_1-u_2\|_{H^{2-\delta}(\Om)}
\end{equation}
for every $u_i\in H^2_0(\Om)$ with $\|u\|_{H^{2}(\Om)}\le R$.
The functional $\Pi(u)$ has the form
\[
\Pi(u)=\int_\Om F(u(x))dx~~\mbox{with}~~ F(s)=\int_0^s f(\xi)d\xi
\]
It follows from
\eqref{phi_condition}
that there exist $\ga<\lambda_1$ and $C\ge 0$ such that
$F(s)\ge -\ga s^2/2-C$ for all $s\in\R$.
This implies \eqref{8.1.1c1}.
\par
{\it Step 2: Von Karman model}. In the case  of the {\it von Karman model} the arguments are more subtle.
We rely on sharp regularity of
Airy stress function \cite{fhlt}  and also Corollary 1.4.5 in \cite{springer}.
\begin{equation*}
\|\Delta^{-2} [u,w]\|_{W^{2,\infty}(\Omega) } \leq C \|u\|_{2,\Omega}\|w\|_{2, \Omega}
\end{equation*}
where $\Delta^2 $ denotes biharmonic operator with zero clamped boundary conditions.
The above yields
\begin{equation*}
\|v(u)\|_{W^{2,\infty}(\Omega) } \leq C \|u\|^2_{2,\Omega}
\end{equation*}
which in turn implies
 that the Airy
stress function $v(u)$ defined in (\ref{airy-1}) satisfies the inequality
\begin{equation}\label{airy-lip}
\| [u_1,v(u_1)]-  [u_2,v(u_2)]\|_\Om\le
C(\|u_1\|_{2,\Om}^2+\|u\|_{2,\Om}\|_2^2)\| u_1-  u_2\|_{2,\Om}
\end{equation}
(see Corollary 1.4.5  in \cite{springer}).
Thus, $f(u)= -[u,v(u)+F_0]$ is locally Lipschitz on $H^2_0(\Om)$.
\par
The potential energy $\Pi$ has the form
\[
\Pi(u)=\frac14\int_\Om\left[ |v(u)|^2 -2([u,F_0]) u\right] dx
\]
and possesses the property in  \eqref{8.1.1c1},
see, e.g., Lemma 1.5.4  \cite[Chapter 1]{springer}.
It is worthwhile to note that the property \eqref{8.1.1c1} is related to the validity of the maximum principle for Monge Ampere equations.  For functions $u \in H^2(\Omega) $ one has
\begin{equation*}
\sup_{\Omega} u  \leq
\sup_{\partial \Omega}u  + \frac{{\rm diam } \Omega}{\sqrt{\pi}} ||[ u, u]||_{L_1(\Omega)}^{1/2}
\end{equation*}
Thus for $u\in H^2(\Omega) \cap H_0^1(\Omega) $  we have (see Lemma 1.5.5 in \cite{springer}):
\begin{equation*}
\max_{\Omega} |u(x)| \le \frac{{\rm diam} \Omega}{\sqrt{\pi}} ||[ u, u]||_{L_1(\Omega)}^{1/2}
\end{equation*}
The above uniqueness property is critical in proving \eqref{8.1.1c1} for any function $ u\in H^2(\Omega) \cap H_0^1(\Omega)$.
\par
{\it Step 3: Berger's model.} One can also see that the Berger model satisfies
\eqref{loc-lip-f} with $\delta=0$ and
 \eqref{8.1.1c1} holds;
for details see \cite[Chapter 4]{Chueshov} and \cite[Chapter 7]{ch-l}.

\subsubsection{Completion of the Proof of Theorem \ref{th:nonlin}}

 For the proof of Theorem \ref{th:nonlin}
it suffices to apply Theorem \ref{abstract} along with
the estimates stated above. These estimates assert that the  hypotheses of Theorem \ref{abstract}  have been verified for all three nonlinear models under consideration.
\par
Concerning strong solutions
we notice first that on the strength of the estimate (\ref{airy-lip}) the domain of $(\bA + \bP) + \cF $  in the von Karman case is the same
as the domain $\cD(\bA + \bP ) $. The same holds for  the other two models.
The local Lipschitz property of the nonlinear terms, along with global bounds on solutions,  allows us to claim the invariance of the domain of the nonlinear  flow.  Thus  for the initial data in $ Y_1$ one has that the solution $y = (\phi, \phi_t; u, u_t)  \in L_{\infty} ( 0, T; Y_1) $ (as in the argument of Theorem~\ref{th:3.9}
we refer to Theorem 1.6~\cite[p.189]{Pazy}).
 This implies
\begin{align*}
& y \in C([0, T]; Y ),~~ \phi_t \in  L_{\infty} ( 0, T; H^1(\R^3_{+})), ~~ u_t \in  L_{\infty} ( 0, T; H_0^2(\Omega)), \nonumber \\
& \phi_{tt} = -2 \phi_{xt} - U^2 \phi_{xx} + \Delta \phi \in L_{\infty} (0, T; L_2(\R^3_{+})),
 \nonumber \\
& u_{tt} = - \Delta^2 u + \gamma[ \phi_t  +U \phi_x]  - f(u)= - \Delta^2 u +U  \gamma[\phi_x]   + \gamma[\phi_t] - f(u)
\in L_{\infty} ( 0, T; L_2(\Omega)),
\nonumber \\
& \mbox{the above implies via elliptic theory and Sobolev's embeddings}
\nonumber \\
& \Delta^2 u \in  L_{\infty} ( 0, T; H^{-1/2}(\Omega))   \rightarrow u  \in L_{\infty} ( 0, T; H^{7/2}(\Omega)),  \nonumber \\
& (U^2-1)\phi_{xx} - \phi_{zz} -\phi_{yy} \in  L_{\infty} ( 0, T; L_2(\R^3_{+})),
~~  \phi_z|_{z=0} = u_t + U u_x \in  L_{\infty} ( 0, T; L_2(\R^2)).
\end{align*}
The above relations imply the regularity properties required from strong solutions.
\par
 The regularity postulated for strong solutions is sufficient in order to define variational forms describing the  solutions. The  existence and uniqueness  of weak solutions follow by viewing
generalized solutions as the strong  limits of   strong solutions.
 This, along with Lipschitz estimates satisfied by nonlinear forces, allows a
passage with the limit on strong solutions.
This completes the proof of Theorem \ref{th:nonlin}.

\medskip\par

In conclusion we note that
the well-posedness results presented in this treatment are a necessary first step in studying long-time behavior of solutions. This can be done without the addition of damping mechanisms (see \cite{chuey} and \cite[Remark 12.4.8]{springer}) or in the presence of control theoretic damping, e.g. boundary or interior dissipation (see \cite{springer,ACC}. In either case, the next step will be to show the existence of global attracting sets for the plate component of the model, and analyze their properties (i.e., compactness, dimensionality, and regularity).

\section{Acknowledgements}

The research conducted by Irena Lasiecka was supported by the grants NSF- DMS-0606682 and AFOSR-FA99550-9-1-0459.
Justin Webster was supported the Virginia Space Grant Consortium Graduate Research Fellowship, 2011-2012 and 2012-2013.

\section{Appendix}
\subsection{Direct Proof of Estimate \eqref{followest} for Fixed Point Statement}
Let  $\bar u \in C^2([0,T];H_0^{2}(\Omega))$ and
let $y(t)=(\phi(t),\psi(t);u(t),v(t))\in C([0,T];Y)$ be a
a mild solution to \eqref{inhomcauchy}.
This implies that $y(t)$ is a
  (distributional) solution
to problem
\begin{equation}\label{sol-distr}
\begin{cases}
(\partial_t+U\partial_x)\phi =\psi& \text{ in } \realsthree_+ \times(0,T),
\\
(\partial_t+U\partial_x)\psi=\Delta \phi -\mu\phi&\text{ in } \realsthree_+\times (0,T),
\\
\Dn \phi = -\big(\partial_tu +U\partial_x w)\cdot \mathbf{1}_{\Omega}(\xb)& \text{ on } \realstwo_{\{(x,y)\}} \times (0,T),
\\
u_{tt}+\Delta^2u=\gamma[\psi]&\text{ in } \Om \times (0,T).\\
u=\Dn u = 0 & \text{ in } \pd\Om \times (0,T).
\end{cases}
\end{equation}
It follows from the trace theorem that
there exists $\eta$ from the class $C^2([0,T];H^{2}(\R^3_+))$ such that
\[
\Dn \eta = -U[\partial_x \bar{u}]_{ext}~~ \text{ on } \realstwo_{\{(x,y)\}} \times (0,T).
\]
Let $\widetilde{\phi}=\phi-\eta$. Then it follows from \eqref{sol-distr}
that
$\widetilde{y}(t)=(\widetilde{\phi}(t),\psi(t);u(t),v(t))\in C([0,T];Y)$
solves (inhomogeneous) problem
\begin{equation}\label{sol-distr-nh}
\begin{cases}
(\partial_t+U\partial_x)\widetilde\phi
=\psi+f_1& \text{ in } \realsthree_+ \times(0,T),
\\
(\partial_t+U\partial_x)\psi=\Delta \widetilde\phi -\mu\widetilde\phi
 +f_2&\text{ in } \realsthree_+\times (0,T),
\\
\Dn\widetilde \phi = -\partial_tu \cdot \mathbf{1}_{\Omega}(\xb)& \text{ on } \realstwo_{\{(x,y)\}} \times (0,T),
\\
u_{tt}+\Delta^2u=\gamma[\psi]&\text{ in } \Om \times (0,T),\\
u=\Dn u = 0 & \text{ in } \pd\Om \times (0,T).
\end{cases}
\end{equation}
where
\[
f_1= -(\partial_t+U\partial_x)\eta,~~~f_2= (\Delta-\mu)\eta.
\]
Problem \eqref{sol-distr-nh} can be written in the form
\begin{equation}\label{absproblem-f}
\widetilde{y}_t=\bA \widetilde{y}+ F(t),~~~ \widetilde{y}(0)=\widetilde{y}_0,
\end{equation}
where $F=(f_1,f_2;0,0)\in C^1([0,T]; Y)$. Therefore
by Corollary 2.5\cite[p.107]{Pazy} for any $\widetilde{y}_0\in \cD(\bA)$
there exists a strong solution $\widetilde{y}$ to \eqref{absproblem-f} in $Y$. This solution possesses the properties
\[
\widetilde{y}\in  C((0,T);Y)\cap C^1((0,T);\cD(\bA)')
\]
and satisfies the relation
\begin{equation}\label{en-rel-tilde}
\|\widetilde{y}(t)\|_Y^2= \|\widetilde{y}(0)\|_Y^2+\int_0^t(F(\tau),\widetilde{y}(\tau))_Y^2 d\tau.
\end{equation}
Now we can return to the original variable
$y(t)=(\phi(t)\equiv\widetilde{\phi}(t)+\eta(t),\psi(t);u(t),v(t))$
and show that \eqref{en-rel-tilde} can be written in the following way
\begin{equation}\label{en-rel-main}
\|y(t)\|_Y^2= \|y(0)\|_Y^2 -2U\int_0^t(\bar{u}_x(\tau),\ga[\psi(\tau)])_{L_2(\Om)} d\tau,
\end{equation}
provided  $\bar{u}\in C^2([0,T];H_0^{2}(\Omega))$ and $\widetilde{y}_0\in \cD(\bA)$.
\par
The integral term in \eqref{en-rel-main} can be estimated as follows:
\begin{align}\label{fixest-prml}
\int_0^t |<\bar{u}_x, \gamma[\psi]>|
\le &~ c_0 \int_{0}^t\left[\|w\|^2_{H^2(\Om)}d\tau +||\gamma[\psi]||^2_{H^{-1/2}(\Omega)}\right]d\tau
\end{align}
One can see that
 the estimate in (\ref{trace-reg-est-M}) can be written
 with the constant $C_T$ which is uniform at any interval, i.e.
  in the form
\begin{equation*}
\int_0^t\|\gamma[\psi](\tau)\|^2_{H^{-1/2} (\R^2)}d\tau\le C_T\left(
E_{fl}(0)+
 \int_0^t\| \Dn \phi(\tau) \|^2d\tau\right)
\end{equation*}
which holds for every $t\in [0,T]$.
Since in our case $\Dn \phi = -\big(v +U\partial_x \bar{u})\cdot \mathbf{1}_{\Omega}(\xb)$, we have that
\begin{align*}
\int_0^t\|\gamma[\psi](\tau)\|^2_{H^{-1/2} (\R^2)}d\tau\le&~ C_T\left(
\|y_0\|_Y^2+
 \int_0^t\left[ \| v(\tau)\|^2 + \|\bar{u}(\tau)\|^2_{H^2(\Om)} \right]d\tau \right)
\end{align*}
for every $t\in [0,T]$.
Therefore \eqref{en-rel-main} and \eqref{fixest-prml}
yield
\[
\|y(t)\|^2_Y\le  C_T\left(\|y_0\|_Y+  \int_0^t \|\bar{u}(\tau)\|^2_{H^2(\Om)}d\tau +
\int_0^t \| y(\tau)\|^2_Y d\tau  \right)
\]
for every $t\in [0,T]$, where $y_0\in Y_1$ and
and $\bar u \in C^2([0,T];H_0^{2}(\Omega))$.
Now we can extend this inequality by continuity
for all $y_0\in Y$ and $w \in C([0,T];H_0^{2}(\Omega))$
to obtain \eqref{followest}.

\end{document}